\documentclass[a4paper,12pt]{amsart}
\usepackage{hyperref, tikz-cd}
\usepackage[T1]{fontenc}

\usepackage{microtype}
\usepackage{amssymb,amsmath,amsfonts,amsthm,bm,latexsym,amscd,verbatim,url,nicefrac,stmaryrd,enumerate,colonequals,dsfont,appendix,stackrel,mathrsfs,cleveref,tikz-cd}
\input xy
\xyoption{all}

\usepackage[normalem]{ulem}

\usepackage{times}
\usepackage{graphicx}
\usepackage{fullpage}

\makeatletter
\newcommand{\oset}[3][0ex]{%
	\mathrel{\mathop{#3}\limits^{
			\vbox to#1{\kern-2\ex@
				\hbox{$\scriptstyle#2$}\vss}}}}
\makeatother

{\bf}{\it}
\newtheorem*{thm*}{Theorem}
\newtheorem{thm}{Theorem}[section]{\bf}{\it}
\newtheorem{prop}[thm]{Proposition}
\newtheorem{lemma}[thm]{Lemma}

\newtheorem{cor}[thm]{Corollary}
\newtheorem{conj}[thm]{Conjecture}
\newtheorem*{conj*}{Conjecture}
\theoremstyle{definition}
\newtheorem{dfn}[thm]{Definition}
\theoremstyle{remark}
\newtheorem{rmk}[thm]{Remark}
\theoremstyle{remark}
\newtheorem{exm}[thm]{Example}

\newtheorem{cons}[thm]{Construction}
\newtheorem{notation}[thm]{Notation}

\newcommand{\A}{\mathbb{A}}
\newcommand{\B}{\mathbb{B}}
\newcommand{\C}{\mathbb{C}}

\newcommand{\F}{\mathbb{F}}
\newcommand{\G}{\mathbb{G}}

\newcommand{\N}{\mathbb{N}}
\newcommand{\Q}{\mathbb{Q}}

\newcommand{\T}{\mathbb{T}}
\newcommand{\Z}{\mathbb{Z}}

\newcommand{\mc}{\mathcal}

\newcommand{\mcC}{\mathcal{C}}
\newcommand{\mcF}{\mathcal{F}}

\newcommand{\mcH}{\mathcal{H}}
\newcommand{\mcD}{\mathcal{D}}

\newcommand{\mcO}{\mathcal{O}}
\newcommand{\mcP}{\mathcal{P}}

\newcommand{\mcX}{\mathcal{X}}
\newcommand{\mcY}{\mathcal{Y}}

\newcommand{\mf}{\mathfrak}

\newcommand{\mfI}{\mathfrak{I}}

\newcommand{\mfP}{\mathfrak{P}}

\newcommand{\mfS}{\mathfrak{S}}
\newcommand{\mfU}{\mathfrak{U}}
\newcommand{\mfX}{\mathfrak{X}}
\newcommand{\mfY}{\mathfrak{Y}}

\newcommand{\ad}{\mathrm{ad}}

\newcommand{\an}{\mathrm{an}}

\DeclareMathOperator{\An}{An}

\DeclareMathOperator{\CAlg}{CAlg}

\DeclareMathOperator{\colim}{colim}

\DeclareMathOperator{\ct}{ct}
\newcommand{\dR}{\mathrm{dR}}

\newcommand{\eff}{\mathrm{eff}}
\DeclareMathOperator{\et}{\acute{e}t}

\newcommand{\logcrys}{\mathrm{log}\text{-}\mathrm{crys}}
\DeclareMathOperator{\divet}{div\acute{e}t}

\DeclareMathOperator{\FDA}{{FDA}}

\DeclareMathOperator{\Frob}{Frob}

\DeclareMathOperator{\FSm}{FSm}
\DeclareMathOperator{\Gal}{Gal}
\newcommand{\GK}{}%

\newcommand{\gr}{\mathrm{gr}}

\DeclareMathOperator{\Hom}{Hom}

\newcommand{\HK}{\mathrm{HK}}
\DeclareMathOperator{\id}{id}

\DeclareMathOperator{\Ind}{Ind}

\DeclareMathOperator{\logFDA}{logFDA}
\DeclareMathOperator{\logDA}{logDA}

\DeclareMathOperator{\map}{map}

\newcommand{\op}{\mathrm{op}}
\newcommand{\perf}{\mathrm{perf}}

\newcommand{\Prl}{{\rm{Pr}^{L}}} 

\newcommand{\Prlo}{{\rm{Pr}^{L}_\omega}}
\newcommand{\Prloo}{{\rm{CAlg}(\rm{Pr}^{L}_\omega)}}

\newcommand{\rig}{\mathrm{rig}}

\newcommand{\RigSm}{\mathrm{RigSm}}

\newcommand{\Sm}{\mathrm{Sm}}

\DeclareMathOperator{\Spa}{Spa}
\DeclareMathOperator{\Spec}{Spec}

\DeclareMathOperator{\Spf}{Spf}
\newcommand{\st}{\mathrm{st}}

\newcommand{\til}{\widetilde}

\DeclareMathOperator{\tr}{tr}

\DeclareMathOperator{\Ch}{{Ch}}

\DeclareMathOperator{\DA}{{DA}}
\DeclareMathOperator{\UDA}{{UDA}}
\DeclareMathOperator{\DM}{{DM}}

\newcommand{\FSch}{\mathrm{FSch}}

\newcommand{\unr}{\mathrm{nr}}

\DeclareMathOperator{\Mod}{{Mod}}

\DeclareMathOperator{\PerfDA}{{PerfDA}}

\DeclareMathOperator{\QCoh}{{QCoh}}

\DeclareMathOperator{\RigDA}{{RigDA}}

\DeclareMathOperator{\pn}{{pn}}

\newcommand{\sst}{\mathrm{ss}}

\DeclareMathOperator{\Sh}{{Sh}}

\DeclareMathOperator{\wRRig}{\widehat{{Rig}}}
\newcommand{\wRigDA}{\wRRig\!\DA}

\DeclareMathOperator{\Ho}{Ho}

\newcommand{\be}{\begin{eqnarray*}}
\newcommand{\ee}{\end{eqnarray*}}
\newcommand{\ben}{\begin{eqnarray}}
\newcommand{\een}{\end{eqnarray}}
\newcommand{\bx}{\begin{eqnarray*}\begin{xy}\xymatrix}
\newcommand{\ex}{\end{xy}\end{eqnarray*}}

\newcommand{\mr}{\mathrm}

\newcommand{\CH}{\mr{CH}}
\newcommand{\CN}{\mr{CN}}
\newcommand{\cyc}{\mr{cyc}}

\newcommand{\lcc}{(\!(}
\newcommand{\rcc}{)\!)}
\newcommand{\e}{\varepsilon}

\begin{document}
	\title{On the $p$-adic weight-monodromy conjecture for complete intersections in toric varieties}

\author[Binda]{Federico Binda}
\address{Dipartimento di Matematica ``F. Enriques'' - Universit\`a degli Studi di Milano}
\email{federico.binda@unimi.it}
\author[Kato]{Hiroki Kato}
\address{{Institut des Hautes \'Etudes Scientifiques}%
}
\email{{hiroki@ihes.fr}
}
\author[Vezzani]{Alberto Vezzani}
\address{Dipartimento di Matematica ``F. Enriques'' - Universit\`a degli Studi di Milano}
\email{alberto.vezzani@unimi.it}

\thanks{The first and third authors are partially supported by the 
 PRIN 20222B24AY ``The arithmetic of motives and $L$-functions''. %
The second author has received funding from Iwanami Fujukai Foundation and the {\it European Research Council} (ERC) under the European Union's Horizon 2020 research and innovation program (grant agreement No.\ 851146).}

\begin{abstract}
We give a proof of the $p$-adic weight-monodromy conjecture for scheme-theoretic complete intersections in projective smooth toric varieties. The strategy is based on Scholze's proof in the $\ell$-adic setting, which we adapt using  homotopical results developed in the context of rigid analytic motives.
\end{abstract}

\maketitle
\setcounter{tocdepth}{1}
\tableofcontents

\section{Introduction}
Let $K$ be a non-archimedean local field of mixed characteristic $(0,p)$ with ring of integers $\mc O_K$ and residue field $k=\F_{q}$, $q=p^a$, 
and let $K_0=W(k)[1/p]$ be the maximal unramified sub-extension of $K$ over $\Q_p$.  We let $\varphi$ denote the lift of the arithmetic Frobenius in $\Gal(K_0/\Q_p$).  Let $Y$ be a proper smooth   scheme over $K$ with a proper flat integral model $\mcY$ over $\mcO_K$. We let ${\mc Y_s}$ denote its special fiber. 
When $\mc Y$ is smooth over $\mc O_K$, the crystalline cohomology groups $H^i_{\mathrm{crys}}({\mcY}_s/W(k))[1/p]$ come naturally equipped   with a $\varphi$-semilinear endomorphism $\Phi$ called the \emph{Frobenius endormorphism}. 
The Weil conjectures for crystalline cohomology proved by Katz--Messing in \cite[Theorem 1]{KM} imply that $H^i_{\mathrm{crys}}(\mc Y_s/W(k))[1/p]$ is pure of weight $i$ (that is, the  eigenvalues of the linear operator $\Phi^a$ have complex absolute value $q^{i/2}$ via any embedding $\bar K\subset\C$). 

When $\mc Y$ is  \emph{semistable} over $\mc O_K$, a similarly well-behaved $p$-adic cohomology theory is given by Hyodo--Kato cohomology (i.e., log crystalline cohomology due to Hyodo--Kato \cite{HK94}), which we denote by $H^i_{\HK}(Y)=H^i_{\logcrys}({\mc Y_s}/W(k)^0)[1/p]$, where ${\mc Y_s}$ here refers to the special fiber of $\mcY$ equipped with its natural log structure, and $W(k)^0$ is the scheme $\Spec W(k)$ equipped with the log structure associated to the homomorphism of monoids $\N\to W(k), 1\mapsto0$. It comes naturally with a $\varphi$-semilinear Frobenius endomorphism $\Phi$ and a nilpotent endomorphism $N$ called the \emph{monodromy operator}, which satisfy the equality $N\Phi=p\Phi N$.

As Beilinson observed \cite{Bei13}, Hyodo--Kato cohomology can be viewed as a cohomology theory on the generic fiber, and it extends to the whole category of {proper and smooth }algebraic varieties over $K$\footnote{%
{In fact, Beilinson extends Hyodo--Kato cohomology even to all  varieties over $K$ (see \cite[\S 1.1]{CNCst} for a quick overview). However, we will not use this extension outside smooth and proper varieties later.}}: he defined a cohomology group $H^i_{\HK}(Y_{\bar{K}})$, which %
is a $\Q_p^{\unr}$-vector space equipped with a Frobenius and a monodromy operator as above, and an action of the Galois group $\Gal(\bar{K}/K)$. Moreover, he showed, as predicted  by Jannsen and Fontaine {\cite{Jan87}, \cite[Conjecture 6.2.1]{Fon94}}, that it  agrees with $p$-adic \'etale cohomology via Fontaine's functor: \cite[Formula 3.3.1]{Bei13} states that we have a canonical isomorphism of $(\varphi,N,G_K)$-modules
\be
H^i_{\HK}(Y_{\bar{K}})\otimes_{\Q_p^{\unr}} B_{\st}\cong H^i_{\et}(Y_{\bar{K}},\Q_p)\otimes_{\Q_p} B_{\st},
\ee
which {(together with its compatibility with his de  Rham period isomorphism of filtered modules from \cite{Bei12})} reproves the $C_\st$-conjecture in $p$-adic Hodge theory proved by Tsuji \cite{Tsu99}, Faltings \cite{Fal02}, Nizio\l \ \cite{Niz08}. %

We now recall that %
Deligne's weight-monodromy conjecture {\cite{Del71}} predicts that the $\ell$-adic \'etale cohomology group $H_{\et}^i(Y_{\bar{K}},\Q_\ell)$  is quasi-pure of weight $i$, i.e., %
on the $j$-th graded quotient $\gr^M_j H^i(Y_{\bar{K}},\Q_\ell)$ of the monodromy filtration $M_\bullet$, the action of a Frobenius-lift is pure of weight $i+j$.  
Modeled on it, the $p$-adic weight-monodromy conjecture states the following.

\begin{conj}[{\cite[Page 347]{Jan87}}]\label{pwmc}
The $i$-th Hyodo--Kato cohomology group $H^i_{\HK}(Y_{\bar{K}})$ is \emph{quasi}-pure of weight $i$, i.e.,  the $j$-th graded quotient $\gr^M_j H^i_{\HK}(Y_{\bar{K}})$ of the monodromy filtration $M_\bullet$ is Frobenius-pure  of weight $i+j$. 
\end{conj}

To our knowledge, the known cases of the $p$-adic weight-monodromy conjecture are essentially of the following kind:  when the dimension of $Y$ is $\leq2$ as shown  by Mokrane \cite[Th\'eor\`eme 5.3, Corollaire 6.2.3]{mokrane} (see also \cite{Nakkajima}) and when  $Y$ is a   $p$-adically uniformized variety, as shown independently by de Shalit \cite{deshalit} and Ito \cite{Ito05}. The $\ell$-adic weight-monodromy conjecture is also known in the above cases \cite{RZ82,Ito05}. Moreover, it has also been proved  in the case where $Y$ is a smooth complete intersection in a projective smooth toric variety by Scholze 
in his  seminal paper \cite{scholze}, by using his theory of perfectoid spaces to reduce the problem to the case of an equal characteristic local field, which had already been proved by Deligne in \cite{Del80} (cf. \cite{Ito05b}).

We note that  the (analogue of the) $p$-adic weight monodromy conjecture in the equal characteristic case is also known: Crew \cite{crew} proved it in the globally defined case, by eventually reducing it to the Weil conjectures for crystalline cohomology  using his $p$-adic analogue of Deligne's method; and building on it, Lazda and Pal  \cite{LP16} proved the general case. %

Considering this, it is natural to ask if Scholze's strategy can be adapted also to reduce the $p$-adic weight monodromy conjecture for complete intersections to the equal characteristic case.  We provide an affirmative answer in regards to this expectation:
\begin{thm}[{\Cref{thm:main}}]\label{main thm}
	Let $K$ be a finite extension of $\Q_p$ and $Y$ be  %
	 a smooth scheme-theoretic complete intersection inside a projective smooth toric variety over $K$. Then the $p$-adic weight-monodromy conjecture holds for $Y$.
\end{thm}
Though Scholze proves the $\ell$-adic case under the slightly weaker assumption that $Y$ is a \emph{set-theoretic} complete intersection, we do need to assume that $Y$ is a \emph{scheme-theoretic} complete intersection (see also Sections \ref{E tub},\ref{remarks}). 
To explain the outline of the proof, we briefly recall how Scholze reduced the problem to the equal characteristic case. 

\subsection{Review of Scholze's proof}
Assume that $Y$ is a smooth complete intersection in a projective smooth toric variety $X_{\Sigma,K}$ over $K$.  We may and do replace $K$ with the perfectoid field $K_\infty$ obtained as the completion of the purely ramified extension $K(\varpi^{1/p^\infty})$, for a uniformizer $\varpi$ of $K$.  
His goal is to realize the cohomology group $H^i(Y)=H^i_{\et}(Y_{\bar K},\Q_\ell)$ as a direct summand of the cohomology group of some proper smooth variety $Z$ of equal characteristic. %
The variety $Z$ is constructed using the natural continuous map of topological spaces $\pi\colon |X^\an_{\Sigma,K^\flat}|\to |X^\an_{\Sigma,K}|$ obtained from the tilting equivalence. %
More precisely, the construction consists of the following three steps:
\begin{enumerate}[(A)]
\item\label{tub} We take, using results of Huber \cite{huber}, \cite{huber-f2}, a small (analytic) open neighborhood $\til Y$ of (the analytification of) $Y$ inside $X_{\Sigma,K}^\an$ such that $H^*(Y)\cong H^*(\til Y)$, 
\item\label{approx} We find, by an approximation argument \cite[Proposition 8.7]{scholze}, a closed algebraic subvariety inside $\pi^{-1}(\til Y)$ defined over $k(\!(\varpi^\flat)\!)$ that has the same dimension as $Y$, and then take $Z$ to be a smooth alteration of it. 
\item\label{tilt} Using the tilting equivalence of the \'etale site \cite[Theorem 7.12]{scholze}, %
we obtain canonical maps
\be
H^i(Y)\cong H^i(\til Y)\to  H^i(Z)
\ee
that  are equivariant with respect to the actions of $G_K\cong G_{K^\flat}$. 
\end{enumerate}
The latter morphism can be shown to be  invertible on the top degree: this follows from the Chern classes formalism. %
Poincar\'e duality then provides  a desired splitting. 

It is tempting to follow the same blueprint in the $p$-adic setting. Note first that Step \eqref{approx} is purely geometric and can be applied to the $p$-adic case verbatim. 
On the other hand, Steps \eqref{tub} and \eqref{tilt} of the above recipe require some  arguments that are particular to $\ell$-adic cohomology, and it is unclear a priori how to adapt them to the $p$-adic situation. 
Note that we cannot expect that this argument works for $p$-adic \'etale cohomology: the equal characteristic case due to Lazda--Pal is not about $p$-adic \'etale cohomology and Huber's existence theorem on a good tubular neighborhood cannot hold (e.g., $H^1$ is infinite-dimensional for the rigid disk). 
With this in mind, we work with Hyodo--Kato cohomology on the generic fiber. Even then, we encounter the following difficulties: in Step \eqref{tilt}, the morphism between the  \'etale sites does not directly induce a map on the Hyodo--Kato cohomology groups; %
in Step \eqref{tub}, Huber's theorem on the existence of a tubular neighborhood is only  available  for $\ell$-adic cohomology. The goal of our paper is to overcome these difficulties using homotopical methods. %

\subsection{Tilting and Hyodo--Kato cohomologies}
In Sections \ref{sec:mot} and \ref{sec:HK}, we focus on giving a $p$-adic version of Step \eqref{tilt}.  %
The key ingredient for this step is the motivic tilting equivalence established in \cite{vezz-fw}, which provides a natural way to discuss relations between tilting and $p$-adic cohomology theories (cf.\ \cite{LBV}). %
It is stated in terms of the theory of motives of rigid analytic varieties $\RigDA(K)$ introduced by Ayoub \cite{ayoub-rig}, which can be applied to our situation since {(the rigid)}  Hyodo--Kato cohomology  can be {extended} to rigid analytic varieties thanks to the work of Colmez--Nizio{\l} \cite{CN19}, and it can be shown to be motivic. The motivic tilting equivalence implies that   any ``well-behaved'' cohomology theory defined for rigid analytic varieties over $K$ can be ``tilted'' to a ``well-behaved'' cohomology theory defined for rigid analytic varieties over $K^\flat$ (and viceversa). %
 In particular, it 
allows us to ``tilt''  Hyodo--Kato cohomology $R\Gamma_\HK$ and obtain a cohomology theory $R\Gamma_\HK^\flat$ on the category of smooth rigid analytic varieties over  $\C_p^\flat$ equipped with a functorial $(\varphi,N)$-structure.  

Following Scholze's strategy, and assuming Step \eqref{tub}, we can again realize $H^i_{\HK}(Y_{\bar K})$ as a direct summand of $R^i\Gamma_{\HK}^\flat(Z^{\an}_{{\C}_p})$ for some proper smooth variety $Z$ of equal characteristic. %
Then our proof is reduced to showing that the new cohomology theory $R\Gamma_\HK^\flat$ satisfies the weight monodromy conjecture (\Cref{wmc for tilted HK}). 
For this purpose, we construct, in the case of semistable reduction, a comparison isomorphism that connects $R\Gamma^\flat_\HK$ to the classical Hyodo-Kato cohomology, which satisfies the weight monodromy conjecture by the above-mentioned result of Crew and  Lazda--Pal. %
\begin{thm}[{\Cref{cor:HKtilt}}]\label{comparison}
Let $Z$ be a smooth rigid analytic variety over  a finite extension $F$ of $k\lcc p^\flat\rcc$ (inside ${\C_p^\flat}$) that admits a semistable formal model with log special fiber $Z_0$. 
Then we have a canonical quasi-isomorphism
\be
R\Gamma_\HK^\flat(Z_{\C_p^\flat})\cong R\Gamma_\HK(Z_0/W(\kappa))\otimes_{W(\kappa)[1/p]}\Q_p^\unr,
\ee 
where $\kappa$ denotes the residue field of $F$. 
\end{thm}
We remark that the good reduction case (i.e., non-log case) has been already proved %
in \cite{vezz-tilt4rig}. {In order to compare  cohomology theories on the generic fiber and on the special fiber, }the key point of \emph{loc.\ cit.\ }is the construction of the ``motivic Monsky--Washnitzer'' functor $\DA(k)\to\RigDA(K)$ from the category of algebraic motives   to that of rigid analytic motives. Its definition is based on the invariance of motives under nilpotent thickenings \cite[Corollaire 1.4.24]{ayoub-rig}, which essentially follows from the localization theorem of Morel--Voevodsky. In \cite{vezz-tilt4rig}, it is shown that the motivic Monsky--Washnitzer functor is compatible with the motivic tilting equivalence. Then the good reduction case is obtained by a suitable $p$-adic realization on $\RigDA(K)$. 

In order to generalize this strategy to the \emph{log} case, we introduce the category of log (formal) motives  with respect to the strict-\'etale topology and rational coefficients,   and show the following:
\begin{thm}[\Cref{FDA=DA}, \Cref{prop:comm0} ]
\begin{enumerate}
\item 	Let $\mfS$ be a quasi-coherent integral log formal scheme of finite Krull topological dimension. The special fiber functor induces an equivalence 
	\be
	\logFDA^{}(\mf S)\stackrel{\sim}{\to} \logDA^{}(\mf S_s),
	\ee
	between the category of log formal motives over $\mf S$ and that of log motives over the special fiber $\mf S_s$. 
	\item Let $K$ be a perfectoid field with residue $k$. Then the diagram
   \be
   \begin{tikzcd}
        \logFDA^{\mathrm{v}}(\mcO_{K}^\times) \arrow[d, "\xi_K"] \arrow[r, "\sim"] & \logDA^{\mathrm{v}}(k^0) &  \logFDA^{\mathrm{v}}(\mcO_{K^\flat}^\times)  \arrow[l, swap ,"\sim"]  \arrow[d, "\xi_{K^\flat}"]\\ 
        \RigDA(K) \arrow[dash, rr, "\sim"] && \RigDA(K^\flat)
    \end{tikzcd}
    \ee
	is commutative up to an invertible natural transformation. 
	\end{enumerate}
\end{thm}

In the theorem above, we denote by $\mcO^\times$ the  formal scheme $\Spf\mcO$ equipped with the natural log structure, and by $k^0$ the log scheme $\Spec k$ equipped with the pullback log structure.   %
The superscript $\mathrm{v}$ refers to the full subcategory of \emph{vertical} log motives (which have a trivial log structure on the generic fiber) and $\xi$ the functor induced by the rigid analytic generic fiber. By putting $K=\C_p$ we obtain the comparison isomorphism in \Cref{comparison} as a realization of this commutative diagram. Thus, our task has been reduced to giving a $p$-adic analogue of Step \eqref{tub}.

\subsection{Existence of a good tubular neighborhood}\label{E tub}
In \Cref{sec:tube} we give a motivic   version of Huber's theorem on the existence of a good tubular neighborhood of $Y$. %
Let $\B_K^n$ denote the $n$-dimensional rigid poly disk over a non-archimedean field $K$. 
\begin{thm}[{\Cref{cor:tn}}]\label{intro:tube}
	Let $X\to S$ be a qcqs smooth morphism  of smooth rigid analytic varieties over $K$. Let $s$ be a $K$-rational point of $S$ and let $X_s$ denote the fiber of $X$ over it. Then, for any sufficiently small open neighborhood $U$ of $s$ that is isomorphic to $\B^N_K$, the natural morphism from the motive of $X_s$ to that of $X\times_SU$ in $\RigDA(K)$ is invertible. %
\end{thm}
We prove this as a consequence of the ``spreading out" property of rigid analytic motives shown in \cite[Theorem 2.8.15]{AGV}.  %
In particular,  we can take a tubular neighborhood that does not change the $\ell$-adic cohomology groups independently on $\ell$, which reproves/generalizes, with a completely different method, a smooth intersection case of the main result of \cite{Ito20}, in which this $\ell$-independence property is proved using the theory of nearby cycles over general bases. 

\subsection{Remarks on the  proof of the main theorem}\label{remarks} In \Cref{sec:proof} we finally put together the above ingredients to prove our main theorem. 
Note that,  since our proof of \Cref{main thm} is motivic, it can be applied to the $\ell$-adic setting as well, in which case it essentially coincides with Scholze's. Nonetheless, note that \Cref{intro:tube} requires the morphism $X\to S$ to be \emph{smooth}, contrarily to what is shown by Huber.  %
This is the reason why our methods only allow us to get the weight-monodromy conjecture for \emph{scheme}-theoretic complete intersections; not for set-theoretic complete intersections as in \cite{scholze} (but, see \Cref{rmk:hope}). 

Finally, we also remark that a motivic approach allows us to define a version of Hyodo--Kato cohomology purely on the generic fiber, without making any reference to log schemes or the log-de Rham Witt complex (\Cref{sec:bond}). This is coherent with Fontaine's initial expectation that the monodromy operator should be defined directly on a $\Q_p^{\unr}$-model of the de Rham cohomology, using rigid analytic techniques \cite[Remarque 6.2.11]{Fon94}. This construction is a consequence of the description of rigid analytic motives in terms of modules over the residue field, proved in \cite{AGV}, by which any compact motive in $\RigDA(K)$ can be seen as a motive over $\G_{m,k}$ (up to a finite field extension). The Frobenius-action and the monodromy operator arise naturally from the motivic nearby cycle functor (paired up with the notion of weight-structures studied by Bondarko), giving a unified motivic version on the $\ell$-adic and $p$-adic Steenbrick complexes (see also \cite{BGS}).  Such an intepretation will be further developed in a future paper.

\section*{Acknowledgments}
We thank K\k{e}stutis \v{C}esnavi\v{c}ius for having pointed out to us \Cref{geoconn} and Joseph Ayoub for the numerous discussions on  \Cref{motmon}. We also thank Veronika Ertl and Wies{\l}awa Nizio{\l} for their precious suggestions on \Cref{subsec:HK}, Fabrizio Andreatta, Teruhisa Koshikawa and Doosung Park for providing useful references on log geometry and $\varphi$-modules, {and the anonymous referee for their useful suggestions.}
The second author also thanks Ahmed Abbes, Kazuhiro Ito, Takeshi Saito, Atsushi Shiho, Kazuki Yamada, Zijian Yao and Yifei Zhao for helpful discussions, and expresses his gratitude to K\k{e}stutis \v{C}esnavi\v{c}ius  for his interest and constant encouragement.

\section{Log motives and rigid motives}\label{sec:mot}
In this section we introduce and study categories of motives that will be relevant throughout the paper.
\subsection{Preliminaries on log structures}
Even though we are ultimately interested in  (semistable) logarithmic schemes over the residue field of $k$, we need some general facts on motives attached to log (formal) schemes. {From now on, all formal schemes are assumed to be  adic of finite ideal type (see \cite[Definitions I.1.1.14 and I.1.1.16]{fujiwara-kato}).  }

		Our general references for log geometry are \cite{Kato89,Ogu} for fine log structures and \cite{Kos20} for  (not necessarily fine) log (formal) schemes.  	In particular, our log structures are always defined on the small \'etale sites, i.e., a log (formal) schemes is a pair $\mfX=(\underline{\mfX},M_{\mfX})$ with $\underline{\mfX}$ being a (formal) scheme and $M_{\mfX}$ being a sheaf of monoid{s} on the small \'etale site $\underline{\mf X}_{\et}$ together with $\alpha\colon M_{\mf X}\to \mc O_{\underline{\mf X},\et}$ that induces {an }isomorphism $\alpha^{-1}(\mc O^*_{\underline{\mf X},\et})\cong \mc O^*_{\underline{\mf X},\et}$. %
		We will only consider quasi-coherent, \emph{integral} log structures (\cite[Section 2]{Kato89}). %
		
		If $\mfX=(\underline{\mfX},M_{\mfX})$ is a log formal scheme, %
        the morphism of \'etale sites $\underline{\mfX}^{\rig}_{\et}\to\underline{\mfX}_{\et}$ induces a log structure $M_{\mf X^\rig}\to \mc O_{\underline{\mf X}^\rig,\et}$ %
		on the generic fiber $\underline{\mf X}^\rig$ of $\underline{\mfX}$ %
		(see \cite[Section II.3.2.9]{fujiwara-kato} for the definition of the structure sheaf). 
		One can find further information on log rigid analytic spaces in \cite{DLLZ19}, though they will not play any  role in our paper. 

		\begin{dfn}\label{def:form}
	Let $\mfS$ be an %
	integral quasi-coherent log (formal) scheme. %
	We say that a morphism $f\colon \mfX\to\mfS$ is   \emph{smooth} if, \'etale locally on $\underline{\mfX}$ and $\underline{\mfS}$, there exists a fine  log structure $M_{0}$ on $\underline{\mf S}$ with a morphism $M_0\to M_{\mfS}$, a %
	log smooth morphism  $(\underline{\mfX},M_{\mfX0})\to (\underline{\mfS},M_0)$, and an isomorphism $(\underline{\mfX},M_{\mfX0})\times_{(\underline{\mfS},M_0)}\mfS\cong \mf X$ over $\mf S$, where the fiber product is taken in the category of integral log (formal) schemes. 

		We let  $\Sm/\mfS$  denote the category of log smooth log (formal) schemes over $\mfS$. 
	This category can be equipped with the strict-\'etale topology, which we denote by $\et$.
\end{dfn}
\begin{prop}\label{smmaps}
For any morphism $f\colon \mfX\to\mfX'$ in $\Sm/\mfS$, \'etale locally on $\underline{\mfX}$, $\underline{\mfX}'$, and $\underline{\mfS}$ there is a fine   log structure $M_0$ on $\underline{\mfS}$ equipped with a morphism $M_0\to M_{\mfS}$, and  a morphism   $(\underline{\mfX},M_{\mfX,0}) \to(\underline{\mfX}',M_{\mfX',0}) $ in $\Sm/(\underline{\mfS},M_0)$ such that $f$ is its pullback along   $M_0\to M_{\mfS}$. 
\end{prop}
\begin{proof}
We may and do suppose that $\underline{\mfS}=\Spf A$ is affine and that $M_\mfS$ has a chart $M_\infty\to A$. Also, we may and do suppose that $g\colon \mfX\to\mfS$ [resp.  $g'\colon \mfX'\to \mfS$] has a model $(\underline{\mfX},M_{\mfX0})$ [resp. $(\underline{\mfX}',M_{\mfX'0})$] over the fine log structure $M_0$ [resp. $M_{0'}$] on $\underline{\mfS}$. We may even assume there is a fine chart $M_0\to N_0$ [resp. $M_0\to N'_{0}$] of this map. {By considering the log structure associated to $M_0\oplus M_{0'}$}, we may suppose that $M_0=M_{0'}$. {By replacing $\underline{\mfS}$ be an \'etale cover}, we may assume that the log structure $M_{\mfS}$ {has a chart $M_\infty$} which is the filtered colimit $\varinjlim M_{i}$ as $M_i$ varies among  fine log structures on $\underline{\mfS}$ over $M_0$ with a map to $M_{\mfS}$.

{
The monoid $M_\infty$ is  the filtered colimit $\varinjlim M_{i}$ as $M_i$ varies among  fine monoids over $M_0$ with a map to $M_{\infty}$. We let $N_i$ [resp. $N_i'$] be the integral pushout $N_0\oplus_{M_0}M_i$ [resp. $N'_0\oplus_{M_0}M_i$] and $N_\infty$ be the integral pushout $N_0\oplus_{M_0}M_\infty$ [resp. $N'_0\oplus_{M_0}M_\infty$] .

We are left to show that a $M_\infty$-linear morphism $f\colon N_\infty\to N'_\infty$ has a $M_i$-linear model $N_i\to N'_i$ for some large enough $i$. We consider the following composite morphism over $M_0$
\be
N_0\to N_\infty\to N'_\infty=\varinjlim N'_i.
\ee
Since $N_0$ is finitely presented over the fine monoid $M_0$ in the sense of  \cite[Lemma I.2.1.9]{Ogu} there exists an index $i$ for which the morphism above factors over $M_0$ as  $N_0\stackrel{f_0}{\longrightarrow}N'_i\to N'_\infty$. By eventually taking the base change over $M_i$ of $N_0$ and renaming the indices,  we may and do assume that $i=0$. Since $N_\infty$ [resp. $N'_\infty$] is the push-out of $N_0$ [resp. $N'_0$] along $M_0\to M_\infty$, we also deduce that the morphism $f_\infty$ is the push-out of $f_0$ along $N_0\to N_\infty$, as wanted. 
}
\end{proof}

We will mostly restrict ourselves to vertical morphisms of log schemes:

				\begin{dfn}[{\cite[7.3]{NakLogEt1}}] We say that a morphism of log (formal) schemes $f\colon\mfX\to\mfS$  is \emph{vertical} if the cokernel of $f^{-1}M_{\mfS}\to M_{\mfX}$ is a sheaf of groups. 
				We denote by $\Sm^{\mathrm v}/\mfS$ the full subcategory of $\Sm/\mfS$ consisting \emph{vertical} and log smooth (formal) schemes over $\mf S$.  

\end{dfn}
\begin{rmk}
\begin{enumerate}
\item For a morphism $\varphi\colon P\to Q$ of monoids, the cokernel of $\varphi$ is a   group if and only if, for every $q\in Q$, there exists $q'\in Q$ and $p\in P$ such that $q=q'+\varphi(p)$. It is also equivalent to the ideal $\emptyset\subset Q$ being the only prime ideal $\mf q\subset Q$ such that $\varphi^{-1}(\mf q)=\emptyset$. In particular, it is equivalent to $P/P^*\to Q/Q^*$ satisfying the same condition. 
\item Vertical morphisms are stable under composition: if $G\to P$ is a morphism of monoids with $G$ a group such that the cokernel is also a group, then $P$ is a group. 
\item Vertical morphisms are stable under pullbacks and taking the rigid analytic fiber (as $\underline{\mf S}$ is adic and of finite ideal type): for a morphism of monoids, the condition for the cokernel to be a group is stable under pushout. %
	\end{enumerate}
\end{rmk}

We will often use the following notation.
\begin{notation}\label{log notation}
	Let $K$ be a non-archimedean field with ring of integers $\mc O_K$ and residue field $k$. Let $\mc O_K^\times$ denote the log formal scheme whose underlying scheme is $\Spf\mcO_K$ and whose log structure is the one associated to the pre-log structure $\mcO_K\setminus\{0\}\to \mc O_K$. 
	We will denote by $k^0$ the log scheme whose underlying scheme is $\Spec k$ and whose log structure is given by the pullback along $\Spec k\to\Spf\mcO_K$.  %
\end{notation}

We will implicitly use the following elementary fact about the log structure on perfectoid rings: 
\begin{lemma}\label{exm:logstructures} 
	Let $K$ be a perfectoid field with tilt $K^\flat$. %
	The multiplicative map $\sharp\colon \mc O_{K^\flat}\cong \lim_{a\mapsto a^p}\mc O_K\to \mc O_K$, or equivalently, the composite of the Teichm\"uller lift $\mc O_{K^\flat}\to A_{\inf}=W(\mc O_{K^\flat})$ and  Fontaine's map $\theta\colon A_{\inf}\to \mc O_K$, induces a chart $\mc O_{K^\flat}\setminus\{0\}\to \mc O_K$ of the log scheme $\mc O_{K}^\times$.  
	In particular, the two log structures on $\Spec k$ induced from $\mc O_K$ and $\mc O_{K^\flat}$ are canonically identified. 
\end{lemma}
\begin{proof}
This follows from the bijectivity of $\sharp\colon (\mc O_{K^\flat}\setminus\{0\})/\mc O_{K^\flat}^*\to (\mc O_{K}\setminus\{0\})/\mc O_{K}^*$. 
\end{proof}
\begin{dfn}\label{sst}
\begin{enumerate}
\item Let $K$ be  a discrete valuation  field and let $\mf X$ be a formal scheme over $\Spf \mc O_K$. 
We say that $\mfX$ is \emph{semistable} if %
if {Zariski}-locally it admits an \'etale $\mc O_K$-morphism to \be
	\Spf(\mcO_K\langle u_1, \ldots, u_n\rangle/(u_1\cdots u_m-\varpi))\ee for some integers $0 \leq m \leq n$, where $\varpi$ is a uniformizer of $K$. The  category  of semistable formal schemes over $\mc O_K$ will be denoted by $\FSch^{\sst}/\mcO_K$.   
	\item Let $C$ be the completion of an algebraic closure of a discrete valuation  field $K$. We say that a formal scheme $\mfX$ over $\Spf\mcO_C$ is  \emph{$K$-semistable}  if  it is basic semistable in the sense of \cite[2.2.1(b)]{CN19} i.e. obtained as the pullback along $\Spf\mcO_C\to\Spf\mcO_{L}$ from a semistable formal scheme over $\mcO_L$ for some finite field extension $L/K$. The  category  of $K$-semistable formal schemes  will be denoted by $\FSch^{\sst}/\mcO_C$. 
\item  We say that a log formal scheme $\mf X=(\underline{\mf X},M_{\mf X})$ over $\mc O_K^\times$ [resp.\ $\mc O_C^\times$] is \emph{semistable} [resp.\ $\mc O_C^\times$] if $\mf X$ is log smooth and vertical over $\mc O_K^\times$ and if $\underline{\mf X}$ is a semistable [resp.\ $K$-semistable] formal scheme over $\mc O_K$ [resp.\ $\mc O_C$]. We denote by $\Sm^{\sst}/\mcO^\times_K$ [resp.\ $\log\FSm^{\sst}/\mcO_C^\times$] the  full subcategory of $\Sm/\mcO^\times_K$ [resp.\ $\Sm/\mcO_C^\times$] consisting of semistable log formal schemes over $\mc O_K^\times$ [resp.\ $\mc O_C^\times$]. 	
\end{enumerate}
\end{dfn}
One could also introduce pluri-nodal or poly-stable versions of the definitions above (cf.\ \Cref{sec:motives+}).
	\begin{rmk}\label{rmk:logF=F}
	 Assume that $K$ is [the completed algebraic closure of] a discrete valuation  field. 
	 The category $\FSch^{\sst}/\mcO_K$ is canonically equivalent to $\Sm^{\sst}/\mcO_K^\times$. %
	 Indeed, if $\mfX \to \mcO_K^\times$ is vertical and log smooth, $\mfX$ is (log) regular by \cite[IV.3.5.3]{Ogu}  and the log structure on $\underline{\mfX}$ is the log structure associated to the pre-log structure $(\mcO_{\underline{\mfX}}[1/p]^\times)\cap \mcO_{\underline{\mfX}}\to \mc O_{\underline{\mf X}}$ by \cite[Proposition 2.6]{Niz06} (or \cite[Prop. 12.5.54]{GR}, cf.\ \cite[Theorem 11.6]{Kat94b}, \cite[Example A.1]{Kos20}). Thus, the equivalence of categories follows (see also  {\cite[III.1.6.2]{Ogu}}).  Note that the same equivalence holds for the pluri-nodal or the poly-stable versions 
	 since both (like the property of being semistable) are properties of the underlying schemes.
	\end{rmk}

We note that the log structure on the residue field $k_C$ is the one induced by the pre-log structure $\Q_{\geq0}\cong |\varpi|^{\Q_{\geq0}}\to C$ sending $a\in \Q_{>0}$ to $0$. On the residue field, it then makes sense to give the following definitions.
	
	\begin{dfn}\label{ss over log pt}
	Let $k$ be a field. For an integral monoid $\Gamma$, whose monoid structure we will write multiplicatively, we let $M_{\Gamma}^0$  denote the log structure on $\Spec k$  associated to $\Gamma\to k$ sending $\gamma\ne1$ to $0$. 
	For an element $\gamma\in\Gamma$, let $\iota_\gamma$ denote the morphism of log schemes   $(\Spec k, M_{\Gamma}^0)\to(\Spec k,M_{\N}^0)$ obtained by the identity on $\Spec k$ and the map $\N\to \Gamma;1\mapsto \gamma$ of monoids. 
\begin{enumerate}
\item		 Let $X=(\underline X,M_X)$ be a log smooth log scheme over $(\Spec k,M_{\N}^0)$. 
		 We say that $X$ is \emph{semistable} if {Zariski}-locally on $\underline X$, the log scheme $X$ admits a (strictly) \'etale $(\Spec k,M_{\N}^0)$-morphism to $k[u_1,\ldots,u_n]/(u_1u_2\cdots u_m)$ equipped with the standard log structure, for some $0\leq m\leq n$ (cf.\ \cite[III.1.8.4]{Ogu}).  
		 We denote by $\Sm^{\sst}/(\Spec k,M_{\N}^0)$ the full subcategory of $\Sm/(\Spec k, M_{\N}^0)$ consisting of semistable log schemes over $(\Spec k,M_{\N}^0)$. 
\item		 Let $X=(\underline X,M_X)$ be a log smooth log scheme over $(\Spec k, M_{\Gamma}^0)$. 
We say that $X$ is $\gamma$-\emph{semistable} if it is isomorphic to the fiber product $X_0\times_{(\Spec k,M_{\N}^0),\iota_\delta}(\Spec k, M_{\Gamma}^0)$ for some  $\delta$ such that $\gamma\in \delta^{\N}$  and some semistable log scheme $X_0$ over $(\Spec k,M_{\N}^0)$. %
We denote by $\Sm^{\sst}/(\Spec k, M_{\Gamma}^0)$ the full subcategory of $\Sm/(\Spec k, M_{\Gamma}^0)$ consisting of $\gamma$-semistable log schemes over $(\Spec k, M_{\Gamma}^0)$. 
    	\end{enumerate}
	\end{dfn}

\begin{rmk}
In the above definition, the only relevant case for us is when $k$ is the residue field of the completion $C$ of an algebraic closure of a discrete valued field $K$. In this case, we will always chose $\gamma=|\varpi|$ where $\varpi$ is a uniformizer of $K$.  From the definitions above, we then obtain that the special fiber of a $K$-semistable log formal scheme over $\mcO_C$ is $|\varpi_K|$-semistable. 
\end{rmk}

		\subsection{Logarithmic motives} 

{For an adic topological ring $A$ with an ideal of definition $I$, we let $A\langle u\rangle$ denote the $I$-adic completion of the polynomial ring $A[u]$, and equip it with the $I$-adic topology. }For a formal scheme $\underline{\mf S}$, we let $\A^1_{\underline{\mf S}}$ [resp. $\G_{m,\underline{\mfS}}$] denote the formal scheme given by $\Spf \mcO_{\underline{\mfS}}\langle u\rangle$ [resp. $\Spf \mcO_{\underline{\mfS}}\langle u^{\pm1}\rangle$]. {
It is (adic and) smooth over ${\underline{\mf S}}$, and %
its associated rigid-analytic space is typically denoted by $\B^1_{\underline{\mf S}^{\rig}}$ [resp. $\T^1_{\underline{\mf S}^{\rig}}$]. 
}
For a log (formal) scheme ${\mf S}$, we let $\A^1_{\mf S}$ [resp. $\G_{m,\mfS}$] denote the log formal scheme $(\A^1_{\underline{\mf S}},p^*M_{\mf S})$ [resp. $(\G_{m,\underline{\mfS}}, p^*M_{\mfS})$] where $p$ is the natural projection to $\mfS$. The following is a straighforward generalization of the classical (infinity-categorical) definition of motives, see e.g. \cite[Definitions 2.1.15 and 3.1.3]{AGV}
		
		\begin{dfn}\label{dfn:mot}Let 
		$\mfS$ 
		be a log (formal) scheme. %
		\begin{enumerate}
		\item
			We let %
			$\logFDA^{\eff}_{}(\mfS)$ be the full monoidal infinity-subcategory  consisting of %
			$\A^1_{\mfS}$-invariant objects in the monoidal stable infinity-category of \'etale hypersheaves  %
			$\Sh^{\wedge}_{\et}(\Sm/\mfS,\Ch\Q)$, where $\et$ denotes the strict-\'etale topology. We let $L_{\A^1}$ be the localization functor (left adjoint to the natural inclusion) $\Sh^{\wedge}_{\et}(\Sm/\mfS,\Ch\Q)\to \logFDA^{\eff}(\mfS)$. 
			\item We denote by $\Q(1)$ the image by $L_{\A^1}$ of the cofiber of the split inclusion of representable sheaves induced by the morphism $1\colon \mfS\to\G_{m,\mfS}$. %
We let  %
$\logFDA^{}(\mfS)$ be its stabilization  with respect to the Tate twist $\otimes\Q(1)$ (i.e., the formal inversion of $\Q(1)$ in the sense in \cite[Definition 2.6]{Rob15}). 
\end{enumerate}
For a log scheme $S=(\underline{S},M_S)$, we write $\logDA^{(\eff)}(S)$ instead of $\logFDA^{(\eff)}(S)$. 	There is a Yoneda functor from $\log\Sm/\mfS$ to $\logFDA(\mfS)$, which we denote by $\mfX\mapsto \Q_{\mfS}(\mfX)$. %
We use the same notation for $\logDA(S)$. For $n\in\Z$, the $n$-th power Tate twist will be denoted by $M\mapsto M(n)$.

We define the category $\logFDA^{\rm v}(\mfS)$ of vertical log formal motives similarly, i.e., using $\Sm^{\mr v}/\mf S$ instead of $\Sm/\mf S$. 
If $K$ is [the completed algerbraic closure of] a discrete valuation  field, we  can also define  the categories %
 $\logFDA^{\sst}(\mcO_K^\times)$ of  semistable log formal motives, with the obvious notation, and if $k$ is its residue field, we can analogously introduce the category  %
  $\logDA^{\mathrm{ss}}(k^0)$ of semistable log motives. %

		\end{dfn}

\begin{rmk}
There is an obvious universal property of motives with respect to functors having strict-\'etale descent, $\A^1$-invariance and inverting the Tate twist: see \cite[Remark 2.8]{LBV}
\end{rmk}
\begin{rmk}Sometimes it would be more natural to consider the \emph{Kummer-\'etale} or the \emph{log \'etale}  %
topology on log schemes rather than the strict-\'etale topology. We remark that in that case one would get a further localization of the category $\logDA$ introduced above. 
Any log smooth vertical morphism over $\mcO_K$  is log \'etale locally polystable if the group of values of the valuation ring $\mcO_K$ is
divisible  (e.g.  when $K$ is algebraically closed) by \cite[Theorem 5.2.16]{ALPTlog}. This is also  true for the   coarser topology $\divet$   (see   \cite[Definition 3.1.5]{logmot}). Sometimes it would also be natural to consider the $(\mathbb{P}^1,\infty)$-localization rather than the $\A^1$-localization to define log motives  (see \cite{logmot}). We do not pursue this approach here. %

\end{rmk}
\begin{rmk}\label{descent}\label{adm}
Arguing as in \cite[Theorem 2.3.4]{AGV}, it is immediate to see that the functor $\mfS\mapsto\logFDA^{(\mr v)}(\mfS)$, $f\mapsto f^*$ where $f^*$ is induced by the pullback, has \'etale hyperdescent.
\end{rmk}

\begin{rmk}\label{hyper}
In \Cref{dfn:mot}, the category of \'etale hypersheaves consists of those presheaves having descent with respect to  \'etale hypercovers (see \cite[Definition 2.3.1 and Remark 2.3.2]{AGV}). If the base $\mfS$ has a finite (topological) Krull dimension (which will be the case in all the relevant situations) then this category coincides with that of \emph{sheaves} $\Sh_{\et}^{}(\Sm/\mfS,\Ch\Q)$ i.e., those having descent with respect to \v{C}ech \'etale hypercovers (see \cite[Lemma 2.4.18]{AGV}). 
\end{rmk}
\begin{rmk}\label{rmk:cptg}
In case  $\mfS$ has a finite (topological) Krull dimension, the category  $\logFDA_{}(\mfS)$ is compactly generated. A class of compact generators is given by motives attached to affine log (formal) schemes that are log smooth over the base. A qcqs morphism between log formal schemes of finite Krull dimension induces a functor $f^*$ which is in $\Prlo$ (see \cite[Proposition 2.4.22]{AGV}).%
\end{rmk}

	\begin{rmk}
		\label{Prloo}
		Arguing as in \cite[Remark 2.1.13]{AGV}, the categories $\logFDA(\mfS)$ and the functors $f^*$ introduced above are in $\CAlg(\Prl)$ i.e. they are  equipped with a symmetric monoidal  structure such that  $\Q_{\mfS}(\mfX)\otimes\Q_{\mfS}(\mfX)\cong \Q_{\mfS}(\mfX'')$ where the fiber product $\mfX''=\mfX\times_{\mfS}\mfX'$ is taken in the category of integral quasi-coherent log (formal) schemes.
	\end{rmk}
\begin{rmk}
We warn the reader that we will denote by $\logFDA^{(\eff)}(\mfS)$ the categories which, in the context of \cite{AGV}, would be denoted as $\logFDA_{\et}^{(\eff)\wedge}(\mfS,\Q)^{\otimes}$. 
\end{rmk}
	\begin{rmk}\label{rmk:psin}
Note that any log scheme strict-\'etale over a vertical [resp.\ semistable] log scheme is vertical [resp.\ semistable]. In particular, the categories 	$\logFDA^{\rm v}(\mfS)$ [resp. $\logFDA^{\sst}_{}(\mcO^\times_K)$ and  $\logDA^{\sst}_{}(k^0)$] are full sub-categories of $\logFDA^{}(\mfS)$  [resp. of $\logFDA^{}_{}(\mcO^\times_K) $ and  $\logDA^{}_{}(k^0)$, respectively]    (see for example \cite[Proposition 3.17]{vezz-fw}). The same is true for the poly-stable or the pluri-nodal versions.
	\end{rmk}

\subsection{Invariance of log motives under the special fiber}
Assume $\mfS$ is a  log formal scheme. We let $\underline{\mfS}_\sigma$ be the special fiber of $\underline{\mfS}$ (i.e. the reduced scheme associated to $\underline{\mfS}/\mfI$ where $\mfI$ is an ideal of definition) and $\mfS_\sigma=(\underline{\mfS}_\sigma,M_{\mfS_\sigma})$ the  log scheme induced by the closed immersion $\underline{\mfS}_\sigma\subset\underline{\mfS}$. 
		The inclusion $\iota\colon {\mfS_\sigma} \to {\mfS}$ induces an adjunction $(\iota^*,\iota_*)$ (see also \cite[Notation 3.1.9]{AGV})
		\be
		\begin{tikzcd}
			\logFDA^{(\eff)}_{}(\mfS)\arrow[rr,shift left=.5ex,"\iota^* "]\arrow[rr,"\iota_*"' ,leftarrow, shift right=.5ex]&& \logDA^{(\eff)}_{}({\mfS_{\sigma}}).
		\end{tikzcd}
		\ee
		The following is a log-variant of \cite[Corollary  1.4.29]{ayoub-rig} (see also \cite[Theorem 3.1.10]{AGV}). %

		\begin{thm}\label{FDA=DA}
			Let $\mfS$ is a quasi-coherent integral log formal scheme. Assume that at least one of the following holds:
			\begin{enumerate}
\item The log structure $M_{\mfS}$ is fine;
\item 			$\underline{\mfS}$ has 	finite topological Krull dimension.
			\end{enumerate}%
		 Then the adjunction $(\iota^*, \iota_*)$ gives an equivalence in $\CAlg(\Prl)$: \be\logFDA^{(\eff)}(\mfS)\cong\logDA^{(\eff)}({\mfS_\sigma}).\ee
		\end{thm}

We will follow closely the proof of  \cite[Corollary 1.4.29]{ayoub-rig} which is in turn based on the classical proof of the localization axiom for motives by Morel-Voevodksy \cite[Theorem 2.21]{MV99}.  We will reproduce  it, borrowing some notation from \cite{Hoy21}, adapting it to our specific setting.  %
	\begin{proof}		As in the classical case, it suffices to prove the (stronger) effective version of the statement (i.e. before Tate stabilization) and, since $\iota^*$ is monoidal, we can prove the claim in $\Prl$. We first prove the statement assuming $M_{\mfS}$ is fine.

	 By the topological invariance of the \'etale site \cite[Th\'eor\`eme 18.1.2]{EGAIV4} and the characterization of log smooth maps \cite[Theorem 3.5]{Kato89}, it is immediate to see that  the functor $\iota^*$ sends a class of  generators to a class of  generators, and that $\iota_*$ commutes with all colimits (see \cite[Lemmas  2.2.4 and 2.2.5]{AGV}).

	In particular, 	in order to show that $\mcF\cong\iota_*\iota^*\mcF$, as the functors $\iota^*$ and $\iota_*$ commute with colimits, we may assume that $\mcF=\Q_{\mfS}(\mfX)$ for some $\mfX$ in $\Sm/\mfS$.  	We remark that the functor $\iota_*$ between the categories of (complexes of $\Q$-linear) presheaves preserves $(\et,\A^1)$-local equivalences. We shall then consider these objects as presheaves, and prove that this morphism is a $(\et,\A^1)$-local equivalence (i.e. it becomes invertible in the motivic category).
		
		To this aim, it suffices to check that for any $\mfS'$ 
		smooth over the base $\mf S$, and any morphism $\mcF'\colonequals\Q_{\mfS}(\mfS')\to \iota_*\iota^*\mcF$ that corresponds to a morphism $s\colon \mfS'_\sigma\to \mfX$, the pullback $\mcF\times_{\iota_*\iota^*\mcF}\mcF'\to \mcF'$ is an equivalence. This morphism (see \cite[Corollaire 4.5.40]{ayoub-th2} or \cite[Corollary 5]{Hoy21}) is the image via a functor $f_\sharp $  (which preserves $(\A^1,\et)$-equivalences) of the morphism $T_{(\mfX',s)}\to\Q$ between presheaves on $\Sm/\mfS'$, where we put $\mfX'\colonequals \mfS'\times_{\mfS}\mfX$  and where $T_{(\mfX',s)}$  is the presheaf sending $\mfP$ to the free $\Q$-module on the set of morphisms $\mfP\to \mfX'$ over $\mfS'$ that factor over the morphism $s'\colon \mfS'_\sigma\to\mfX'$ induced by $s$ on the special fiber.  In particular, it suffices to show that $T_{(\mfX',s)}\to\Q$ is an equivalence. Up to renaming the objects, we then fix a section $s\colon \mfS_\sigma\to\mfX$ of $\mfX\to \mfS$ and we are left to prove that the presheaf $T_{(\mfX,s)}$ on $\Sm/\mfS$ is $(\et,\A^1)$-equivalent to $\Q$. 
		
		By \cite[Proposition 4.10]{Kato89}, for some strict \'etale cover $(\mfU_i)$ of $\mfS$, the composite $(\mfU_i)_\sigma\to \mf S_\sigma\overset{s}{\to}\mf X$ factors into a strict morphism ${s'_i}\colon (\mfU_{i})_\sigma\to\mfX'_i $  followed by a log \'etale morphism ${e_i}\colon  \mfX'_i\to\mfX$. In particular, the section $s_i\colon (\mfU_i)_\sigma\to%
		\mfX\times_{\mfS}\mfU_i$ of $\mfX\times_{\mf S}\mf U_i\to\mfU_i$ factors as $
		(\mf U_i)_\sigma\to (\mf U_i\times_{\mf S}\mf U_i)_\sigma\to \mf X'_i\times_{\mf S}\mf U_i\to \mfX\times_{\mfS}\mfU_i,
		$
		where the first morphism is the diagonal. 
		Since the statement that we want to prove is \'etale local on $\mfS$ (see \cite[\'Etape 1 of the proof 2.4.21]{ayoub-rig}), we may and do assume that  the section $s$   decomposes into a strict closed immersion ${s}'\colon \mfS_\sigma\to \mfX' $, followed by a log \'etale morphism ${e}\colon  \mfX'\to \mfX$. Since $s'$ targets the strict locus of the morphism $\mfX'\to \mfS$, which is an open of $\mfX$ (see, for example,  \cite[Proposition 3.19]{olsson}), by eventually shrinking $\mfX'$ to this open subscheme,  we may and do  assume that $\mfX'\to\mfS$ is strict. %

		We now claim that  the natural morphism $T_{(\mfX',s')}\to T_{(\mfX,s)}$ is an isomorphism.  We can argue as follows: let $f\colon \mfP\to\mfS$ be in $\Sm/\mfS$ and let $h\colon\mfP \to\mfX$ be a section of $T_{(\mfX,s)}$ on $\mfP$, i.e. a morphism over $\mfS$ which factors over $s$ on the special fiber.  Note that the following diagram commutes.%
		\be 
		\begin{tikzcd}
		\mfP_\sigma \arrow[d] \arrow[r] & \mfX'_\sigma\arrow[d]\\
		\mfS_\sigma \arrow[ru, "s'"] & \mfX_\sigma\arrow[l]
		\end{tikzcd}
		\ee We have to show that $h$ determines uniquely  a section of $T_{(\mfX',s')}$ on $\mfP$. Indeed, let  $\mfP'$ be the  fiber product 
		$ \mfP' = \mfP \times_{\mfX}\mfX' $
		in the category of   log formal schemes: it is log \'etale over $\mfP$, and the diagram above gives a section $\tau_\sigma$ on the special fiber of the morphism  $g\colon \mfP' \to\mfP$, which composed with $\mfP_\sigma' \to\mfX'_\sigma$ agrees with $s'\circ f_\sigma$. It is then enough to show that $\tau_\sigma$ can be lifted uniquely to a section $\tau$ of the log \'etale morphism $g$. This is true as soon as the log \'etale toposes on $ \mfP$ and $\mfP_\sigma$ are equivalent. Since the morphism $\mfP_\sigma \to\mfP$ is a strict universal homeomorphism, we can use the main result of \cite{vidal} (together with   \cite[Corollaries I.2.2.5 and IV.3.1.11]{ogus-ct} to get rid of the saturated assumption) to conclude. %

		By what we have just proved, we may and do replace $\mfX$ with $\mfX'$ and hence assume that the morphism $\mfX\to\mfS$ is strict, and hence, %
		 induced by a smooth morphism between the underlying schemes. In this case, we may conclude the statement of the theorem as in \cite[\'Etape 3 of Proposition 1.4.21]{ayoub-rig} or  \cite[Proposition  4.5.42]{ayoub-th2}. This concludes the proof under the assumption that $M_{\mfS}$ is fine.
		
		We now move to the non-fine case. Fix a quasi-coherent integral log formal scheme $\mf S=(\underline{\mfS},M_{\mfS})$ with finite Krull dimension, and note that the functor $\logFDA((-))\to\logDA((-)_\sigma)$ is a morphism between  {hypersheaves} on $\mfS_{\et}$ with values in $\Prlo$  (see Remarks \ref{adm} and \ref{hyper}). In order to prove it is an equivalence, by \cite[Proposition 2.8.1]{AGV} it suffices to check it is invertible on stalks. To this aim, we fix a geometric point $\bar{s}\to\underline{\mfS}$ and we remark that (see the notation and the result of \cite[Proposition 2.5.8]{AGV})\be\varinjlim_{\mfU}\logFDA(\mfU)  \cong \logFDA( (\mfU)_{\mfU}),\ee
		where $\mfU$ runs among \'etale neighborhood of $\bar{s}$ in $\mfS$. By definition of smooth morphisms and \Cref{smmaps}, the \'etale topos on $\varinjlim \Sm/\mfU$ is equivalent to the \'etale topos on $\varinjlim \Sm/(\underline{\mfU},M_{\mfU}^f)$ where we now let $\mfU$ vary among \'etale neighborhoods of $\bar{s}$ and  $M_{\mfU}^f$ vary among fine log structures on $\mfU$ with a morphism to $M_{\mfU}$. We then deduce that
		\be\varinjlim_{\mfU}\logFDA(\mfU)  \cong \logFDA( ((\underline{\mfU},M^f_{\mfU}))_{(\underline{\mfU},M_{\mfU}^f)})\cong \varinjlim_{(\underline{\mfU},M_{\mfU}^f)} \logFDA((\underline{\mfU},M_{\mfU}^f)).\ee
		As the analogous equivalences hold for $\logDA((-)_\sigma)$, we can deduce the claim from the fine case.
	\end{proof}

 By \Cref{rmk:psin} we can restrict the previous equivalence to vertical/semistable motives.
\begin{cor}\label{iotaiseq}Let $\mfS$ be a quasi-coherent integral log formal scheme with a fine log structure, or with an underlying formal scheme of finite  topological Krull dimension.  
The equivalence of \Cref{FDA=DA} restricts to an equivalence 
\be
	 \logFDA^{\mathrm{v}}(\mfS)\cong\logDA^{\rm v}({\mfS_\sigma}).
	\ee
Analogously, if $K$ is a complete non-archimedean field, we obtain an an equivalence 
\be\logFDA^{\sst}(\mcO_K^\times)\cong\logDA^{\sst}(k^0).\ee
The same is true if one restricts to poly-stable or pluri-nodal motives on both sides.\qed
\end{cor}

\begin{rmk}
We note that, based on the proof above, one can get a general ``localization triangle" (see \cite[Proposition 2.2.3(2)]{AGV}) for (formal) \emph{log \'etale}  log motives: in this case, one can work log \'etale locally and hence replace $\mfX$ with $\mfX'$ (in the notation of the proof above) without invoking any special property of the closed immersion chosen. 
\end{rmk} 

\begin{rmk}
In the proof of \Cref{FDA=DA},	we do not use $\Q$-coefficients in a crucial way. Part (1) holds for the categories of motives with coefficients in any commutative ring spectrum $\Lambda$. The proof of part (2) needs some admissibility conditions on the base to ensure that the categories involved are in $\Prlo$. We expect this condition to be superfluous. %
\end{rmk}

\begin{rmk}\label{interpret}
Let $K$ be a local field. We remark that \Cref{iotaiseq} gives in particular  natural equivalences
\be\begin{tikzcd}
\logFDA^{\sst}(\mcO^\times_K)\arrow[r, "\sim"]& \logDA^{\sst}(k^{\log})& \arrow[l, "\sim"']\logFDA^{\sst}{(}(W(k)^0),
\end{tikzcd}
\ee
where $W(k)^0$ is the log formal scheme structure on $\Spf W(k)$ induced by $\N\mapsto 0$. This is the motivic interpretation of the definition of  Hyodo--Kato cohomology \cite{HK94,GK05} which can be canonically defined on the right-most category, and hence on the left-most category as well (see \cite[Theorem 0.1]{GK05}). 
\end{rmk}

\begin{cor}\label{cor:cont}Let $(\mfS_i)_i$ be a cofiltered inverse system of quasi-compact and quasi-separated log formal schemes with affine transition morphisms, and let $\mfS$ be its limit.  Assume that each log formal scheme in $\{\mfS_i,\mfS\}$ has a fine log strucure or a finite topological Krull dimension. Then the canonical functor defines an equivalence 
\be
\varinjlim	\logFDA(\mfS_i)\cong\logFDA(\mfS).
\ee
\end{cor}
\begin{proof}
We may prove the analogous statement for the special fibers, where it follows from the equivalence of the \'etale toposes on $\Sm/\mfS_\sigma	$ and on $\varinjlim \Sm/\mfS_{i,\sigma}$.
\end{proof}
\begin{exm}\label{ex:cont}
If $C$ is the completion of an algebraic extension of a  local field $K$ then $\logFDA(\mcO_C^\times)\cong\varinjlim\logFDA(\mcO_L^\times)$ as $L$ varies among finite extensions of $K$ inside $C$. 
\end{exm}

\subsection{The log Monsky--Washnitzer functor and its compatibility with tilting}
In this subsection, we show that the ``log Monsky--Washnitzer functors'' are compatible with the motivic tilting equivalence from \cite{vezz-fw}, which is a log generalization of the main result of \cite{vezz-tilt4rig}. 

First of all, we introduce the categories of rigid analytic motives. 
\begin{dfn}\label{RigDA}
Let $S$ be a rigid analytic space (in the sense of \cite[Definition II.2.2.18]{fujiwara-kato}). 
We denote by $\RigDA^{(\eff)}(S)$ the category $\RigDA^{(\eff)\wedge}_{\et}(S,\Q)$ of (effective) hypercomplete rigid analytic motives over $S$ with respect to the \'etale topology (\cite[Definition 2.1.15 and Remark 2.1.18]{AGV}). It has a natural symmetric monoidal structure. 
For a non-archimedean field $K$, we write $\RigDA^{(\eff)}(K)$ for the category $\RigDA^{(\eff)}((\Spf \mc O_K)^\rig)$. 
We denote by $\RigDA(K)^{\ct}$ the full subcategory of $\RigDA(K)$ of compact objects (or, equivalently, of fully dualisable objects by \cite[Theorem 2.31]{ayoub-rig} and \cite{riou-dual}). %
\end{dfn}

\begin{rmk}
In \cite{AGV} the category $\RigDA^{(\eff)}(S)$ endowed with its monoidal structure is typically denoted by $\RigDA_{\et}^{(\eff)\wedge}(S,\Q)^{\otimes}$. In the following, we will only consider rigid analytic spaces with a finite Krull dimension, for which the non-hypercomplete version of the definition $\RigDA^{}(S,\Q)$ coincides with the one given above \cite[Lemma 2.4.18]{AGV}.%
\end{rmk}

Let $\mf S=(\underline{\mf S},M_{\mf S})$ be a log formal scheme. 
Since $\underline{\mfS}$ is adic and of finite ideal type,  we have the rigid space $\underline{\mf S}^\rig$ associated to $\underline{\mf S}$ (see \cite{fujiwara-kato}). 
We assume, for simplicity, that the log structure $M_{\mf S}$ is trivial on the rigid generic fiber, in which case we write $\mf S^\rig$ for $\underline{\mf S}^\rig$. 
Then the functor $\Sm^{\rm v}/\mf S\to \Sm/\mf S^\rig\colon \mfX\mapsto \mf X^\rig=\underline{\mfX}^{\rig}$ %
induces an adjunction (see also \cite[Notation 3.1.12]{AGV})
\be
\begin{tikzcd}
	\logFDA^{\mathrm{v} \,(\eff)}_{}(\mfS)\arrow[rr,shift left=.5ex,"\xi_{\mf S} "]\arrow[rr,"\chi_{\mf S}"' ,leftarrow, shift right=.5ex]&& \RigDA^{(\eff)}_{}({{\mfS}^\rig}).
\end{tikzcd}
\ee
We call the following composite the log Monsky--Washnitzer functor
\be
\logDA^{\mr v\,(\eff)}(\mf S_\sigma)\cong \logFDA^{\mr v\,(\eff)}(\mf S)\to \RigDA^{(\eff)}(\mf S^\rig). 
\ee

Let $K$ be perfectoid field with tilt $K^\flat $.  
Recall that the main theorem of \cite{vezz-fw} gives a \emph{motivic tilting equivalence} $\RigDA(K)\cong \RigDA(K^\flat)$ (see \cite[Theorem 3.12]{BV}). 

According to \Cref{exm:logstructures}, the log schemes $\mcO_K^\times$ and $\mcO_{K^\flat}^\times$ (see \Cref{log notation}) induce the same log structure on the residue field. 
Then the compatibility of the motivic tilting equivalence and the Monsky--Washnitzer functor can be stated as follows, which generalizes \cite[Theorem 3.2]{vezz-tilt4rig} (see also \cite[Proposition 5.11]{LBV}). 
\begin{prop}
	\label{prop:comm0}
	Let $K$ be perfectoid field with tilt $K^\flat $.  
	The following diagram commutes up to a canonical invertible natural transformation.
	\be
	\xymatrix{
		\logFDA^{\mathrm{v}}(\mcO^\times_K)\ar[r]&\RigDA_{}(K)\ar@{-}[dd]^{\sim}\\
		\logDA^{\mathrm{v}}(k^0)\ar@{-}[u]_{\sim}\ar@{-}[d]^{\sim}\\
		\logFDA^{\mathrm{v}}(\mcO^\times_{K^\flat})\ar[r]&\RigDA_{}(K^\flat)\\
	}
	\ee
\end{prop}

For the proof, we use some notation of \cite{LBV}. 

\begin{rmk}\label{rmk:laxfixedpoints}Let $K$ be a perfectoid field. The base change along the Frobenius defines an endofunctor $\RigDA(K^\flat) \xrightarrow{\varphi^*} \RigDA(K^\flat)$, equipped with a natural invertible transformation $\id \Rightarrow \varphi^*$. In particular, for every motive $M$ there is a natural equivalence $M\xrightarrow{\sim} \varphi^* M$, which can be used to define a natural functor $\RigDA(K^\flat)\to \RigDA(K^\flat)_\omega^{h\varphi}$ to the Frobenius fixed points (see \cite[Section 2.3]{LBV} and the notations therein). This functor is compatible with the rigid-analytic generic fiber functor $\xi$ and with the restriction functor $\iota^*$ in the sense that the diagram
	\be\begin{tikzcd}
	\logDA^{\rm v}(k^0) \arrow[d] & \arrow[l, swap, "\iota^*"] \logFDA^{\rm v}(\mcO^\times_{K^\flat}) \arrow[d] \arrow[r, "\xi"] & \RigDA(K^\flat) \arrow[d] \\
	\logDA^{\rm v}(k^0)^{{\mathrm{lax}}-h\varphi}_\omega &\arrow[l] \logFDA^{\rm v}(\mcO_K^\times)^{{\mathrm{lax}}-h\varphi}_\omega \arrow[r]& \RigDA(K^\flat)^{h\varphi}_\omega
	\end{tikzcd}
	\ee
	is commutative. Here the left vertical functors are the canonical one{s} to the lax homotopy fixed points of the Frobenius computed in $\Prlo$. Informally the objects of $\mcC^{\mathrm{lax}-hF}$ for a compactly generated presentable infinity-category $\mcC$ equipped with an  endofunctor $F\colon \mcC\to \mcC$ with a right adjoint are pairs $(X, \alpha)$, where $X$ is an object of $\mcC$ and $\alpha$ is a map $X\to FX$. Note that the homotopy fixed points are a localization of the lax fixed points using e.g. \cite[Proposition 5.5.3.17]{lurie}. 
	See \cite[Definition II.1.4]{NS} for more details. The subscript $\omega$ stands for the full subcategory generated under filtered colimits by compact objects.
\end{rmk}

\begin{proof}[Proof of \Cref{prop:comm0}]%
	We recall that the equivalence $\RigDA(K)\cong\RigDA(K^\flat)$ can be obtained by passing over the Fargues-Fontaine curve as in \cite[Corollary 5.14]{LBV}. That is,  for a fixed affinoid open neighborhood $U=\Spa (B_{[0,\varepsilon]},B_{[0,\varepsilon]}^+)$ of $x^\flat=\Spa K^\flat$ in $\Spa W(\mcO_{K^\flat})^{\an}$ containing $x^\sharp=\Spa K$  the tilting equivalence is given as the composition of the first line in the following diagram, where we use the notation of \cite[Section 5.1]{LBV}:
	\be
	\xymatrix{
		\RigDA(K^\flat)\ar[r]&	\RigDA(K^\flat)^{h\varphi}_\omega\cong\RigDA(U)^{hj^*}_\omega\ar[r]^-{x^{\sharp*}}&\RigDA(K)\\
		\logFDA^{\mathrm{v}}(\mcO^\times_{K^\flat})\ar[r]\ar[u] & \logFDA^{\mathrm{v}}(\mcO^\times_{K^\flat})^{{\mathrm{lax}}-h\varphi}_\omega\ar[u]\\
		\logDA^{\mathrm{v}}(k^0)\ar[r]\ar@{-}[u]^-{\sim} & \logDA^{\mathrm{v}}(k^0)^{{\mathrm{lax}}-h\varphi}_\omega\ar@{-}[u]^-{\sim}
	}
	\ee
	in which the squares commute (see \Cref{rmk:laxfixedpoints}).  We now show that the composite $\logDA^{\mathrm{v}}(k^0)\to\RigDA(K)$ obtained in the diagram above is induced by the natural one. To this aim, it suffices to consider the following commutative $\varphi$-equivariant diagram 
	\be
	\xymatrix{
		&\logFDA^{\mathrm{v}}(\mcO_K^\times)\ar[rr]\ar@{-}[dl]_{\sim}&&\RigDA(K)  %
		\\
		\logDA^{\mathrm{v}}(k^0)&\logFDA^{\mathrm{v}}((A_{\inf},M_{K^\flat}))\ar[r]\ar@{-}[l]_-{\sim}\ar[u]\ar[d]&\logFDA^{\mathrm{v}}((B_{[0,\varepsilon]}^+,M_{K^\flat}))\ar[r]&\RigDA(U)\ar[u]\ar[d]\\
		&\logFDA^{\mathrm{v}}(\mcO^\times_{K^\flat})\ar[rr]\ar@{-}[ul]_{\sim}&&\RigDA(K^\flat)
	}
	\ee
	in which the functors on the left are invertible by means of \Cref{iotaiseq}.
\end{proof}

	\begin{cor}
	\label{prop:comm} %
	Let $C$ be the completion of an algebraic closure of a discrete valuation field $K$ with uniformizer $\varpi$ and residue field $k$. Let $C^\flat$ be its tilt, which is the  completion of an algebraic closure of the discrete valuation field  $k(\!(\varpi^\flat)\!)$. %
	The following diagram commutes, up to a canonical invertible natural transformation.
	\be
	\xymatrix{
		\logFDA^{\sst}(\mcO^\times_{C})\ar[r]&\RigDA_{}(C)\ar@{-}[dd]^{\sim}\\
		\logDA^{\sst}(k^{0})\ar@{-}[u]_{\sim}\ar@{-}[d]^{\sim}\\
		\logFDA^{\sst}(\mcO^\times_{C^\flat})\ar[r]&\RigDA_{}(C^\flat)\\
	}
	\ee
\end{cor}
\begin{proof}
	The diagram is connected to the one of \Cref{prop:comm0} via fully faithful functors.
\end{proof}

The relation between rigid motives and semistable log motives is  encoded by the following.
	\begin{prop}\label{prop:sst}%
	Let $C$ be the completion of an algebraic closure of a discrete valuation field $K$ of residual characteristic $p>0$. The generic fiber functor $\xi$ exhibits  $\RigDA^{(\eff)}(C)$ [resp. $\RigDA^{(\eff)}(C^\flat)$] as the rig-\'etale localization of the category $\logFDA^{\sst}(\mcO^\times_{C})$ [resp. $\logFDA^{\sst}(\mcO^\times_{C^\flat})$]. 
\end{prop}
\begin{proof}We use that $\RigDA(C)\cong\varinjlim\RigDA(L)$ runs over finite extensions $L/K$ inside $C$ \cite[{Theorem 2.8.15}]{AGV}. Then, by \cite[Notation 2.5.5]{AGV}, $\RigDA(C)$ is the category of motives over the \'etale topos $\varinjlim\RigSm/L$. 
It now suffices to show that the \'etale topos on $\RigSm/K$ is equivalent to the rig-\'etale topos on potentially semistable formal models over $\mcO_K$. This is the content of \cite[Proposition 2.8]{CN19}. The same proof works in the equi-characteristic case.
\end{proof}

In light of \Cref{prop:sst}, we may restate \Cref{prop:comm}  as follows.
\begin{cor}\label{cor:locloc}Let $C$ be the completion of an algebraic closure of a discrete valuation field. 
	Via the following natural equivalence  
	\be
	\begin{tikzcd}
		\logFDA^{\sst}(\mcO_{C^\flat}^\times) \arrow[r, "\sim"] &  \logDA^{\sst}(\bar{k}^{0}) & \arrow[l, "\sim"'] \logFDA^{\sst}(\mcO^\times_{C}), 
	\end{tikzcd}
	\ee the rig-\'etale localization on $\logFDA^{\sst}(\mcO_{C^\flat}^\times)$  corresponds to the rig-\'etale localization on $\logFDA^{\sst}(\mcO_{C}^\times)$.\qed%
\end{cor}
\begin{rmk}
	The equivalence displayed in  \Cref{cor:locloc}, obtained from \Cref{iotaiseq}, looks like a ``stronger" form of the motivic tilting equivalence (of \cite{vezz-fw}) since we are comparing  $\RigDA(\C_p)$ and $\RigDA(\C_p^\flat)$ \emph{before} the two rig-\'etale localizations over $\C_p$ and over $\C_p^\flat$. Nonetheless, a priori it is not clear that such localizations agree, as they may not be expressed in terms of a Grothendieck topology defined  intrinsically on the special fibers.  
\end{rmk}

\section{The  motivic Hyodo--Kato realization}\label{HK realization}\label{sec:HK}
The aim of this section is to recall from \cite{CN19} the construction of (overconvergent) Hyodo--Kato cohomology on the rigid generic fiber and to prove that its tilt, which is a cohomology theory on  rigid analytic varieties over $\C_p^\flat$, is compatible with the classical Hyodo--Kato cohomology of Gro\ss e-Kl\"onne in the semistable case (\Cref{cor:HKtilt}). %
We will give both a ``geometric'' realization (that is, over $\C_p$) and an ``arithmetic'' realization (that is, over a local field). We note that only the former will be used in the proof of our main theorem. 
\subsection{A motivic geometric Hyodo--Kato realization}\label{subsec:HK}
Let $K$ be a complete discrete valuation field with a perfect residue field $k$ and $C$ be the completion of an algebraic closure $\bar{K}$ of $K$, whose residue field is denoted by $\bar k$.  We let $K_0$ denote the  field $W(k)[1/p]$ viewed as a subfield of $K$ and $K_0^{\unr}$ denote the maximal unramified extension of $K_0$ inside $\bar{K}$. 
 
We collect some definitions  %
on Hyodo--Kato cohomology, %
following Gro\ss e-Kl\"onne \cite{GK05}, Beilinson \cite{Bei13}, Ertl--Yamada \cite{EY19, EYcb} and Colmez--Nizio\l \cite{CN19}, reinterpreted using a motivic language. Hyodo--Kato cohomology on the rigid generic fiber defined in \cite{CN19} can be thought of as a rigid analytic analogue of  motivic statements of \cite{DN}, which are based on \cite{NekovarNiziol}.

Though there is a way to define Hyodo--Kato cohomology purely in terms of the generic fiber (see \Cref{sec:bond}), we follow in this section the classical approach using log structures. %

\begin{dfn}[{\cite[1.15]{Bei13}}]\label{dfn phi N}
	Let $\varphi\colon K_0\to K_0$ [resp. $\varphi\colon K_0^{\unr}\to K_0^{\unr}$] denote the endomorphism induced by the absolute Frobenius $a\mapsto a^p$ on the residue field. 
	\begin{enumerate}
		\item A \emph{$\varphi$-module over $K_0$} [resp. \emph{over $K_0^{\unr}$}] is a $K_0$-vector space [resp. a $K_0^{\unr}$-vector space] $D$ equipped with a $\varphi$-semilinear endomorphism $\varphi\colon D\to D$. 
		\item A \emph{$(\varphi,N)$-module over $K_0$} [resp. \emph{over $K_0^{\unr}$}] is a $\varphi$-module $D$ over $K_0$ [resp.\ over $K_0^\unr$] equipped with a linear endomorphism $N\colon D\to D$ satisfying $N\varphi=p\varphi N$. 
		We denote by $\mcD_{\varphi,N}(K_0)$ [resp. $\mcD_{\varphi,N}(K_0^{\unr})$] the DG-derived category of $(\varphi,N)$-modules over $K_0$ [resp. over $K_0^{\unr}$]. 
	\end{enumerate}
\end{dfn}
\begin{rmk}
The category $\mcD_{(\varphi,N)}(K_0^{\unr})$ is equivalent to the category $\varinjlim \mcD_{(\varphi,N)}(W(k')[1/p])$ computed in $\Prl$ as $k'$ runs among finite extensions of $k$ inside $\bar k$. More explicitly, a compact object is given by a bounded complex of finite dimensional $(\varphi,N)$-modules over $k'$ i.e. by  some compact object of $\mcD_{\varphi,N}(W(k')[1/p])$ for a sufficiently large $k'$. %
\end{rmk}

We use \Cref{log notation} for finite extensions of $K$ and $C$. 
In particular, for a finite extension of $k'$ of $k$ inside $\bar k$, we let ${k'}^0$ [resp.\ $\bar k^0$] denote the scheme $\Spec k'$ [resp.\ $\Spec \bar k$] equipped with the log structure associated to $\N\to k';1\mapsto 0$ [resp.\ $\Gamma=(\mc O_C\setminus\{0\})/\mc O_C^*\to \bar k$ sending $\gamma\ne1$ to $0$].

\begin{dfn}\label{dfn:HK}\label{GK0} 
	Let $X=(\underline{X},M_X)$ be a semistable log scheme over $k^0$. We denote by $R\Gamma^{}_{\HK}(X/W(k))$ the  Hyodo--Kato cohomology complex $R\Gamma^{\rig}_{\HK}(X,X)$  defined as in \cite[Definition 3.17 and Remark 3.19]{EY19} (see also \cite[Remark 5.3]{CN19} and  \cite[Section 1.3.3]{EYcb}).  It is  an object in $\mcD_{(\varphi,N)}(K_0)$, whose cohomology groups will be denoted by $H^i_{\HK}(X)$. %
\end{dfn}
\begin{rmk}\label{GKiscrys}
\begin{enumerate}
\item The complex above is quasi isomorphic to the log rigid cohomology complex defined by Gro\ss e-Kl\"onne in \cite[Lemma 1.4]{GK05}. %
\item \label{crys}By \cite[Section 3.11]{GK05}, if $\underline{X}$ is proper, the cohomology groups $H^i_\HK(X)$ are canonically isomorphic to the crystalline Hyodo--Kato cohomology groups defined in \cite[Theorem 5.1]{HK94}. 
\end{enumerate}
\end{rmk}

For a finite extension $k'$ of $k$ inside $\bar k$ and an integer $d\ge1$, let ${k'}^{(d)}$ denote the log scheme ${k'}^0$ viewed as a $k^0$-log scheme via the morphism ${k'}^0\to k^0$ induced by the inclusion $k\subset k'$ and the map $\N\to \N;1\mapsto d$. 
We always regard $\bar k^0$ as a ${k'}^{(d)}$-log scheme via the morphism $\iota_d\colon\bar{k}^0\to {k'}^{(d)}$ induced by the inclusion $k'\subset \bar{k}$ and the map $\N\cong|p|^{\frac{1}{d}\N}\subset \Gamma=(\mc O_C\setminus\{0\})/\mc O_C^*$.
For a semistable log scheme $X'$ over ${k'}^{(d)}$, we denote by $R\Gamma^{(d)}_{\HK}(X'/W(k'))$ the complex $R\Gamma^{}_{\HK}(X'/W(k'))$ equipped with the monodromy operator $N^{(d)}=\frac{1}{d}N$, where $N$ is the usual monodromy operator (cf.\ \cite{NekovarNiziol} and \cite[Section 3.1]{EY19}). 
Note that, if $X'$ comes from a semistable log scheme $X$ over $k^0$, i.e., $X'\cong X\times_{k^0}{k'^{(d)}}$, then we have a canonical $(\varphi,N)$-equivariant quasi-isomorphism 
\be
R\Gamma^{(d)}_{\HK}(X\times_{k^0}{k'}^{(d)}/W(k'))\cong R\Gamma^{}_{\HK}(X/W(k))\otimes_{W(k)[1/p]}W(k')[1/p].
\ee

\begin{rmk}
	The case $d=1$ of the above isomorphism is known as ``unramified base change" for Hyodo--Kato cohomology \cite[Paragraph 5.3.1(i)]{CN19}, whereas the case $d>1$ corresponds to ``ramified base change'' \cite[Section 3A]{NekovarNiziol}.
\end{rmk}

\begin{dfn}
\label{GK2}%
Let $X=(\underline{X},M_X)$ be a $|p|$-semistable log scheme over $\bar{k}^0$ (see \Cref{ss over log pt}). %
Let $R\Gamma^{}_{\HK}(X)$ denote the complex defined by
\be
\varinjlim R\Gamma^{(d)}_{\HK}(X_i/W(k_i))\otimes_{W(k_i)[1/p]}K_0^{\unr},
\ee
where the colimit is taken over the filtered category of objects $X_{i}/k_i^{(d_i)}$ with $d_i\ge1$ an integer, $k_i$ a finite extension of $k$ inside $\bar k$, and $X_i$ a semistable log scheme over $k_i^{(d_i)}$ together with an isomorphism $X_i\times_{k_i^{(d_i)}}\bar k^0\cong X$. Since the transition maps in the colimit are $(\varphi,N)$-equivariant (and even quasi-isomorphic), the complex $R\Gamma_\HK(X)$ is naturally equipped with a $(\varphi,N)$-structure and defines a functor $R\Gamma_\HK\colon \Sm^{\sst}/\bar k^0\to \mc D_{\varphi,N}(K_0^\unr)^\op$.
\end{dfn}

\begin{rmk}
	We  note that, thanks to the existence of  Hyodo--Kato comparison map  (in this setting the statement of \cite[Theorem 0.1]{GK05} suffices) the functor \be R\Gamma_{\HK}^{\GK}\colon \Sm^{\sst}/\bar{k}^0\to\mcD_{(\varphi,N)}(K_0^{\unr})\ee satisfies $\A^1$-invariance, \'etale descent  and Tate-stability (see also the proof of \Cref{CN0}) and hence produces a functor \be R\Gamma^{\GK}_{\HK}\colon\logDA^{{\sst}}(\bar{k}^0)\to \mcD_{{(\varphi,N)}}(K_0^{\unr})^{\op}.\ee 
\end{rmk}

We can extend this functor to log motives over $\mcO_C^\times$ as follows.

\begin{prop}\label{CN0}The composite functor
	\be
	\Sm^{\sst}/\mcO_C^\times\to \Sm^{\sst}/\bar{k}^0\stackrel{R\Gamma_{\HK}}{\longrightarrow}\mcD_{(\varphi,N)}(K_0^{\unr})
	\ee
	has rig-\'etale descent, $\A^1$-invariance and Tate stability. %
\end{prop}
The proposition is essentially a formality given the comparison theorems of Colmez--Nizio\l. Nonetheless, we  briefly recall these constructions, in order to set some notation. 

\begin{proof}
Note that %
 the functor $\mf X\mapsto R\Gamma_\HK(\mf X_\sigma)$ %
 satisfies Zariski descent. 
Thus, in order to prove it has \'etale descent, we may 
take a rig-\'etale  cover $\mfU\to\mfX$ in $\FSch^{\sst}/\mcO_{\C_p}$ underlying a rig-\'etale cover of affine semistable weak formal schemes (see \cite[Proposition 2.13]{CN19}) and show that $R\Gamma_{\HK}^{\GK}(\mfX_\sigma)$ is the limit of the Cech nerve associated to the cover $\mfU$. 

We consider the quasi-isomorphism 
\be
R\Gamma_\HK(\mfX_\sigma){\otimes}_{K_0^{\unr}}C\cong R\Gamma^\dagger_\dR(\mf X_C)
\ee  
from \cite[Formula 5.16]{CN19} (note that the complex $R\Gamma_\HK(\mfX_\sigma)$ is a perfect complex under our assumptions, so that the tensor $\widehat{\otimes}^R_{K_0^{\unr}}$ of \emph{loc. cit.}  is merely a scalar extension).  

Since $R\Gamma^\dagger_{\dR}$ has \'etale descent (see \cite[Proposition 5.12]{vezz-MW}) we conclude that the complex on the left is the limit of the Cech nerve associated to the cover $\mfU$, as wanted.  

Similarly, ${\A}^1$-invariance follows from the $\B^1$-invariance of the %
 overconvergent de Rham complex. %
Tate stability follows from the explicit computation $H^1_{\HK}(\G_{m,\bar k^0})\cong K_0^\unr(1)$. 
		\end{proof}
	
	\begin{dfn}
\begin{enumerate}
\item 
	By \Cref{prop:sst} and \Cref{CN0}, the functor 
	\be
	\Sm^{\sst}/\mcO_C^\times\to \Sm^{\sst}/\bar{k}^0\stackrel{R\Gamma_{\HK}}{\longrightarrow}\mcD_{(\varphi,N)}(K_0^{\unr})^\op
	\ee
	 induces a functor $\RigDA(C)\to \mc D_{\varphi,N}(K_0^\unr)^\op$, which we again denote by $R\Gamma_\HK$. 
		We will use the same notation for the  following induced functors
\[
\logDA^{\sst}(\bar{k}^0)\stackrel{\xi}{\to}\RigDA(C)\stackrel{R\Gamma_{\HK}}{\longrightarrow}\mcD_{(\varphi,N)}(K_0^{\unr})^\op\]
\[\DA(C)\stackrel{\An^*}{\to}\RigDA(C)\stackrel{R\Gamma_{\HK}}{\longrightarrow}\mcD_{(\varphi,N)}(K_0^{\unr})^\op,
\]
where $\An^*$ denotes the functor induced by analytification. 
	\item 	We let $R\Gamma_{\HK}^\flat$ denote the composite functor
	\be
		R\Gamma_{\HK}^\flat\colon\RigDA(C^\flat)\cong \RigDA(C)\stackrel{R\Gamma_{\HK}}{\longrightarrow}\mcD_{(\varphi,N)}(K_0^{\unr})^\op,
		\ee
		where the first equivalence is the motivic tilting equivalence given in \cite[Theorem 7.26]{vezz-fw}. 
		We will use the same notation for the   induced functor
\be
		\DA(C^\flat)\stackrel{\An^*}{\to}\RigDA(C^\flat)\stackrel{R\Gamma^\flat_{\HK}}{\longrightarrow}\mcD_{(\varphi,N)}(K_0^{\unr})^\op.
\ee
\end{enumerate}
\end{dfn}

	\begin{rmk}\label{diagram of realizations}
\Cref{CN0} gives us a commutative diagram of realizations:
		\be\begin{tikzcd}
				\logFDA^{\sst}(\mcO_{C}^\times)\arrow[r, "\sim"] \arrow[d, "\xi"]& 	\logDA^{\sst}(\bar{k}^0)  \arrow[d, "R\Gamma^{\GK}_{\HK}"] \\
				\RigDA(C)  \arrow[r, "R\Gamma^{}_{\HK}" ] & \mcD_{\varphi,N}(K_0^{\unr})^{\op},
			\end{tikzcd}
		\ee	
		where the left vertical map is the logarithmic Monsky--Washnitzer functor. Recall again from \Cref{prop:sst} that $\xi$ is a  localization with respect to the rig-\'etale topology. 
	\end{rmk}
	
	\begin{rmk}The above definition of Hyodo-Kato cohomology on the rigid generic fiber is essentially borrowed from \cite[Section 5.3.1]{CN19}. %
Their cohomology theory  also has \'etale descent (by construction), $\B^1$-invariance, and Tate stability, and hence induces a realization functor
\be
		R\Gamma_{\HK}^{\CN}\colon \RigDA(C)\cong\RigDA^{\dagger}(C)\to\mcD_{(\varphi,N)}(K_0^{\unr})^{\op}.
\ee
		Here, $\RigDA^{\dagger}(C)$ denotes the category of overconvergent rigid analytic motives (see \cite[Definition 4.18]{vezz-MW}), which is equivalent to $\RigDA(C)$ via the canonical functor sending a dagger variety to the underlying rigid variety \cite[Theorem 4.23]{vezz-MW}. 
		
As they agree  on  semi-stable models and both satisfy rig-\'etale descent, the  functors $R\Gamma_{\HK}^{\CN}$ and $R\Gamma_{\HK}^{}$ (of \Cref{CN0}) agree. 
	\end{rmk}
		
				\begin{rmk}\label{kunneth}
			The realization $R\Gamma^{}_{\HK}\colon\RigDA(C)^{\ct}\to \mcD_{\varphi,N}(K_0^{\unr})^{\op}$	 is monoidal. Indeed, the first category is generated under finite colimits by motives of the form $\Q_{C}(X^{\an})$ with $X/C$ a smooth and proper algebraic variety by \cite[Th\'eor\`eme 2.5.35]{ayoub-rig}, and the monoidal structure is the one extended by the Day convolution (see \cite[Remark 2.1.6]{AGV}) $\Q_{C}(X)\otimes\Q_{C}(X')=\Q_{C}(X\times X')$. 
              In particular, monoidality of $R\Gamma^{}_{\HK}$ {can be tested} on these motives, and in this case the formula $R\Gamma_{\HK}(X)\otimes R\Gamma_{\HK}(X')\cong R\Gamma_{\HK}(X\times X')$ is the usual K\"unneth formula for the algebraic Hyodo--Kato cohomology (see e.g. \cite[Lemma 2.21]{DN}). %
		\end{rmk}

\begin{cor}\label{cor:HKtilt}
Let $F$ be a finite extension of $k\lcc p^\flat\rcc$ inside $C^\flat$, where $p^\flat\in C^\flat$ is the element defined by a system of $p$-th power roots of $p$. 
Let $\kappa$ denote the residue field of $F$ and $e$ the ramification index of the finite extension $F/k\lcc p^\flat \rcc$. 
Let $\mf Z$ be a semistable formal scheme over $\Spf\mc O_F$ equipped with the natural log structure. 
Let $Z$ denote the rigid generic fiber and $Z_0$ the special fiber as a log scheme. 
Then we have a canonical quasi-isomorphism 
\be
R\Gamma^\flat_{\HK}(Z_{C^\flat})\cong R\Gamma_{\HK}^{(e)}(Z_0/W(\kappa))\otimes_{W(\kappa)[1/p]}K_0^{\unr}
\ee
in $\mc D_{\varphi,N}(K_0^\unr)$, where $R\Gamma^{(e)}_\HK$ is as in \Cref{GK2}.  
\end{cor}

\begin{proof}
By \Cref{prop:comm0}, we have $\Q(Z_{C^\flat})^\sharp\cong \xi\Q(Z_0\times_{\kappa^{(e)}}\bar k^0)$. 
Then, by the definitions of $R\Gamma_\HK^\flat$ and $R\Gamma_\HK$ (cf.\ \Cref{diagram of realizations}), the left hand side is canonically isomorphic to $R\Gamma_\HK(Z_0\times_{\kappa^{(e)}}\bar k^0)$, which is then identified with the right hand side by \Cref{GK2}. 
\end{proof}

\begin{rmk}
Note that the normalization on the monodromy operator does not alter the monodromy filtration on the cohomology groups of  %
 $R\Gamma_{\HK}^{(e)}(Z_0/W(\kappa))$.%
\end{rmk}

\subsection{Abstract properties of realizations on algebraic motives}
In this section, we recall from \cite{CD-Weil,DM} that any monoidal algebraic realization is %
automatically equipped with some formal structures of a Weil cohomology, such as Poincar\'e duality and a formalism of Chern classes. 
For such structures on  Hyodo--Kato cohomology, see also %
\cite[Section 5]{NekovarNiziol}, \cite{EY20}, \cite{LP16}. %

\begin{dfn}Let $\kappa$ be a field. 
	We let $\DA(\kappa)^{\ct}$ be the full subcategory of compact objects (or, equivalenlty, of fully dualisable objects, see \cite{riou-dual}) in $\DA(\kappa)=\DA_{\et}(\kappa,\Q)=\DM_{\et}(\kappa,\Q)$. The category $\DA(\kappa)^{\ct}$ is equivalent to Voevodsky's category of geometric motives \cite{voe-tri}, \cite[Proposition 8.3, Th\'eor\`eme B.1]{ayoub-etale}. 
\end{dfn}

Let $\kappa$ be a field and $\mcC$ be  a {$\Q$-linear} symmetric monoidal, compactly generated infinity-category %
with a complete $t$-structure and a $\Q$-linear $t$-exact conservative functor $f\colon \mcC\to\mcD(\Lambda)$ for some field extension $\Lambda/\Q$. Let  $R\Gamma\colon \DA(\kappa)\to \mc \mcC^{\op}$  be a (cohomological) realization  functor, i.e., a $\Q$-linear functor in $\Prlo$. 
We assume that its restriction to compact objects {$R\Gamma^{\ct}$ is monoidal}. We will denote by $H^i$ the $i$-th cohomology of $R\Gamma$. 

\begin{rmk}
Note that, by \Cref{kunneth} and the monoidality of the analyitification functor \cite[Proposition 2.2.13]{AGV}, the algebraic Hyodo--Kato realizations $R\Gamma_{\HK}\colon\DA(C)\to \mc D_{\varphi,N}(K_0^\unr)$ and $R\Gamma_\HK^\flat\colon\DA(C^\flat)\to \mc D_{\varphi,N}(K_0^\unr)$ fit in this framework, and all the following properties can be deduced for such algebraic  Hyodo--Kato (co)homologies.
\end{rmk}

\begin{thm}[{Poincar\'e duality; \cite[Theorem 4.3.2]{voe-tri}, \cite[Theorem 1]{CD-Weil} \cite[Theorem 3.11]{ayoub-ICM}}]\label{poincare}
Let $X$ be %
a proper smooth $\kappa$-scheme purely of dimension $d$. 
Then the motive $\Q(X)$ has  a strong dual  in $\DA(\kappa)$ %
given by $\Q(X)(-d)[-2d]$. %
Thus, 
$R\Gamma(X) $ is strongly dualizable with dual $R\Gamma(X)(d)[2d] $, and hence we have a perfect pairing in $\mc C$%
 \be H^i(X)\otimes H^{2d-i}(X)(d)\to 1.\ee
\end{thm}

Now we recall that {(oriented)} motivic realizations come naturally equipped with a theory of Chern classes and cycle classes (cf.\ {\cite[Section 2.3]{CD-Weil}}).

Let $R\Gamma^{\ct\vee}$ be the composition of $R\Gamma^{\ct}$ with the (canonical) dual endofunctor on $\DA(\kappa)^{\ct}$. Since $\DA(\kappa)=\Ind(\DA(\kappa)^{\ct})$ {\cite[Lemma 5.3.2.9]{HA}} we can extend this realization formally to a monoidal  functor that preserves colimits (i.e. to a functor in $\Prloo$, see \cite[Lemma 5.3.5.8]{lurie}):
\be
R\Gamma^\vee\colon \DA_{}(\kappa)^{}\to\mcC.
\ee	
This corresponds to the associated \emph{homological} realization. Informally, if $M$ is a colimit of  compact objects $M=\colim K_i$ then its image is $\colim R\Gamma(K_i^\vee)$. 

Let $R\Gamma^\vee_*$ denote a right adjoint of $R\Gamma^\vee$ and $\mathcal{E}$ be the object $ R\Gamma^\vee_*1$. It is a $\mathbb{E}_\infty$-ring object in $\DA(\kappa,\Lambda)$ (see {\cite[Proposition 2.5.5.1]{SAG}}) that represents the  cohomology theory $f\circ R\Gamma$ on compact objects. In particular, we can apply  the construction in \cite[2.1]{DM} which yields the following.%

\begin{prop}[{Cycle classes}] \label{chern}%
For a  quasi-compact smooth scheme $X$ over $\kappa$ and for each integer $n\ge0$, we let 
\be
\cyc\colon \CH^n(X)\to H^{2n}(X)(n).
\ee
denote the map $\sigma_{\mathcal{E}}$ defined in \cite[(2.1.3.b)]{DM}. 
These maps are compatible with pullback and pushforward. 
\end{prop}
\begin{proof}
Follows from \cite[2.1.3 and 2.1.6.(4)]{DM} and \cite[Proposition 4.2.3 and Corollary 4.2.5]{voe-tri}.
\end{proof}

The following result gives a nonzero criterion which is used in the proof of our main theorem, following Scholze. 
\begin{cor}\label{nonvanishing}
Let $h\colon Z\to Y$ be a morphism of projective smooth schemes over $\kappa$ with $Z$ geometrically irreducible of dimension $d$. 
Assume that the push-forward $h_![Z]$ is a positive cycle class \cite[\S12]{Fulton}. 
Then the natural map $H^{2d}(Y)\to H^{2d}(Z)$ is a nonzero map. %
\end{cor}
\begin{proof}
By \Cref{chern}, we have the following commutative diagram:
\bx{
\CH^d(Y)\ar[r]^\cyc\ar[d]^{h^!} &H^{2d}(Y)\ar[d]
\\ \CH^d(Z)\ar[r]^\cyc\ar[d]^{\deg} & H^{2d}(Z)\ar[d]^{\tr}
\\ \Z \ar[r] &\Lambda.
}\ex
Since the trace map is bijective (\Cref{dimax}), it suffices to show that the composite $\deg\circ h^!$ is nonzero. 
Let $\mc L$ be an ample line bundle on $Y$. 
Then, by \cite[Proposition 8.3.(c)]{Fulton}, we have $\deg h^!c_1(\mc L)^d=\deg (c_1(\mc L)^d\cap h_![Z])$. The latter is nonzero by \cite[Lemma 12.1]{Fulton}. 
\end{proof}

\begin{rmk}\label{nonvanishing2}
Note that if $h\colon Z\to Y$ is proper with a scheme theoretic image $Z'$ of dimension $d=\dim Z$, then $h_![Z]$ is $\deg(\kappa(Z)/\kappa(Z'))\cdot [Z']$, which is a positive cycle class.
\end{rmk}

\begin{rmk}[Dimension axiom]
	\label{dimax}{Suppose now $\kappa$ is the completion of an algebraic closure of a discrete valuation field of residue characteristic $p$} %
	and assume that the trace map induces a bijection  {$H^{2d}(X)(d)\cong\Lambda$} whenever $X$ is a proper smooth variety of semistable reduction of dimension $d$. This is the case for  Hyodo--Kato cohomology theories $R\Gamma$, $R\Gamma^\flat$  {by \Cref{GKiscrys}\eqref{crys} and \Cref{cor:HKtilt}}. 
	
	For a variety $f\colon Z\to \Spec \kappa$, let $H_{c}^*(Z)$ be the cohomology groups with compact support attached to $Z$ i.e. the cohomology groups of $R\Gamma(f_*f^!1)$. By the localization triangle, alterations and induction on the dimension, %
	 one deduces that the trace map $H_c^{2d}(Z)(d)\to \Lambda$ is bijective, whenever $Z$ is a geometrically connected variety of dimension $d$.%
\end{rmk}

\subsection{Arithmetic Hyodo--Kato cohomology via motives}\label{sec:motives+}
The reader who is only interested in the proof of the main theorem can safely skip this section, which is devoted to a definition of an \emph{arithmetic} motivic Hyodo--Kato realization in {equi-characteristic} and mixed characteristic, and their comparison. To this aim, we first recall some properties of the categories of rigid analytic motives \emph{with good reduction}. 
\begin{dfn}
	We let $\RigDA_{\gr}(K)$ be the full subcategory of $\RigDA(K)$ generated, under colimits, shifts and twists, by 	motives ``of good reduction'', i.e. those attached to varieties $\mfX_\eta$   where $\mfX$ is a smooth formal scheme over $\mcO_K$.  
\end{dfn}

		Despite the name, the category generated by motives of good reduction  may contain motives of varieties which do not have good reduction: for instance, it contains all varieties with semi-stable reduction (see \Cref{ssisgr}). As a matter of fact, if $K$ is algebraically closed, then $\RigDA_{\mathrm gr}(K)\cong\RigDA(K)$. This is part of the following proposition, in which we sum up some results of \cite{AGV}.

		\begin{prop}\label{prop:RigDAgr1}
			Let $L$ be  the completion of an algebraic extension of $K$  with residue field $k_L$ and let $C$ be the completion of an algebraic closure of $K$. 
			\begin{enumerate}
				\item We have $ \RigDA_{\gr}(L)\cong\varinjlim\RigDA_{\gr}(K')$ as $K'/K$ varies among finite extensions inside $L$.
				\item If $L/K$ is totally ramified, then the pullback induces a canonical equivalence $\RigDA_{\gr}(K)\cong\RigDA_{\gr}(L)$. In particular we have  $\RigDA_{\gr}(L)\cong\RigDA_{\gr}(L_0)$ with $L_0=W(k_L)[1/p]$ being the  completion of the maximal unramified extension of $K_0$ inside $L$. 
				\item The functor $K\mapsto \RigDA_{\gr}(K)$ has descent with respect to the \'etale site of unramified field  extensions and \be\RigDA_{\gr}(W(\bar{k})[1/p])\cong \varinjlim\RigDA_{\gr}(W(k')[1/p])\cong \RigDA_{\gr}(C)=\RigDA(C).\ee%
				\item Assume $L$ is perfectoid. Then the following diagram commutes.
				\be
				\begin{tikzcd}
					\RigDA_{\gr}(L)\arrow[r] \arrow[d , no head, "\sim"]& \RigDA(L)\arrow[d , no head, "\sim"]\\
					\RigDA_{\gr}(L^\flat)\arrow[r] & \RigDA(L^\flat)
				\end{tikzcd}
				\ee
			\end{enumerate}
		\end{prop}
		\begin{proof}By means of \cite[Theorem 3.3.3(1)]{AGV}, the category $\RigDA_{\gr}(K)$ is canonically equivalent to the category of $\chi1$-modules in $\DA(k)$ with $\chi$ being the right adjoint to the Monsky--Washnitzer functor $\xi\colon\DA(k)\to\RigDA(K)$ (see for example \cite[Notation 3.1.12]{AGV}). We then deduce property (1) from \cite[Theorem 3.5.3]{AGV}, and property (3) from \cite[Propositions 3.5.1 and 3.7.17]{AGV}. Property (4) is already shown in \cite{vezz-tilt4rig}. We are left to show (2). In light of (1) we may assume $L$ to be a finite extension of $L_0$. 
			The base change functor $\Mod_{\chi_01}(\DA(k_L))=\RigDA_{\gr}(L_0)\to \RigDA_{\gr}(L)=\Mod_{\chi1}(\DA(k_L))$ corresponds to a base change along a map of algebras $\chi_01\to\chi1$. We remind that a choice of a uniformizer gives rise to an equivalence $\chi1\cong1\oplus1(-1)[-1]\cong\chi_01$ (see \cite[Remark 3.8.2]{AGV}). Under these identifications, the map of algebras $\chi_01\to\chi1$ corresponds to the  map $\left(\begin{smallmatrix}
				1&0\\0&e_L\end{smallmatrix}\right)$ with $e_L\neq0$ as computed in  \cite[1.4]{AyoAbridged},  \cite[3.4.14-3.5.12]{ayoub-th2} and is therefore invertible  (we use that our coefficient ring is $\Q$). 
		\end{proof}

		\begin{rmk}\label{rmk:UDA}%
		Recall that the cohomological motive $q_*1$ of $q\colon \G_m\to k$ is canonically equivalent to $1\oplus 1(-1)[-1]$ (see e.g., \cite[Scholie 1.4.2, Subsection 1.5.3, and Th\'eor\`eme 2.3.75]{ayoub-th1}). 
			In the previous proof, we remarked in particular that a choice of a uniformizer determines an equivalence $\chi 1\cong q_*1$, and that the category $\RigDA_{\gr}(K)$ is equivalent to the category $\Mod_{\chi1}(\DA(k))	\cong\Mod_{q_*1}(\DA(k))$. This implies that, after the choice of a uniformizer, 	$\RigDA_{\gr}(K)$  can be alternatively thought as the subcategory $\UDA(k)$ of $\DA(\G_{m,k})$, generated under colimits by motives coming from $k$ (see \cite[Corollaries 15.12 and 15.14]{spitzweck}).
		\end{rmk}

	\begin{rmk}\label{K_0 str}
		The equivalence \be\RigDA(C)^{\ct}\cong%
		 \varinjlim\RigDA_{\gr}(W(k')[1/p])^{\ct}\ee already shows that the (overconvergent) de Rham cohomology of a rigid analytic variety comes equipped with a $K_0^{\unr}$-structure, and the equivalence $\RigDA(C)\cong\Mod_{\chi1}(\DA(k))$ shows that it also has a $\varphi$-structure.  Also the monodromy can be build formally from \Cref{prop:RigDAgr1}, see \Cref{sec:bond}.%
	\end{rmk}
		\begin{rmk}\label{eK1}
			The coefficient $e_K$ appearing in the proof above is the ramification index of $K$. We shall see (in \Cref{eK2}) that this is compatible with the presence of the normalization factor on the monodromy.%
		\end{rmk}

		We now show that varieties of semistable reduction lie in $\RigDA_{\gr}(K)$. We start by recalling the following definition from \cite{berk-contr}.

		\begin{dfn}\label{def:ss}
			Let $L$ be a complete non-archimedean valued field with pseudo-uniformizer $\varpi$. A formal scheme $\mfX$ over $\Spf\mcO_L$ is  
			\emph{pluri-nodal} if Zariski-locally it can be written as part of a sequence \be \mfX=\mfX_d\stackrel{f_{d-1}}{\to}\mfX_{d-1}\stackrel{f_{d-2}}{\to}\cdots\to\mfX_1\stackbin{f_0}{\to}\Spf \mcO_L\ee in which each transition map $f_i$ is, \'etale locally on source and target, given by a composition $\Spf A_{i+1}=\mfX_{i+i}\stackrel{e_i}{\to}\Spf A_i\langle u,v\rangle/(uv-a_i)\to \Spf A_i$ with $e_i$ \'etale and $a_i$ invertible in $A_i[\varpi^{-1}]$. The  category  they form will be denoted $\FSch^{\pn}/\mcO_L$. 
		\end{dfn}
		\begin{rmk}
			We remind the reader that semistable formal schemes are pluri-nodal. The same is true for \emph{polystable} formal schemes (in the sense of  \cite{berk-contr}).
		\end{rmk}
		The following fact appears in \cite{ayoub-rig}.
		\begin{prop}\label{ssisgr}Let $L$ be a complete non-archimedean valued field and		let $\mfX$ be in $\FSch^{\pn}/\mcO_L$. Then the motive $M(\mfX_\eta)$  lies in $\RigDA_{\gr}(L)$.
		\end{prop}
		\begin{proof}
			It suffices to follow  the Steps 3-4 of the proof of \cite[Theorem 2.5.34]{ayoub-rig} or Step 4 of the proof of \cite[Proposition 3.7.17]{AGV} that we briefly recall. 	We fix a sequence%
			\be \mfX=\mfX_d\stackrel{f_{d-1}}{\to}\mfX_{d-1}\stackrel{f_{d-2}}{\to}\cdots\to\mfX_1\stackbin{f_0}{\to}\Spf \mcO_L\ee as in the definition. %
			We may show that $\mfX_\eta$ lies in $\RigDA_{\gr}(L)$ by induction on the cardinality of the set  of indices $i$  for which $f_i$ is not smooth, knowing that in case it is empty, then $\mfX_\eta$ lies in $\RigDA_{\gr}(K)$ by definition. We now assume this set is not empty. By descent, we may work locally on $\mfX$ and hence assume that whenever $f_i$ is smooth, then $a_i=1$. This shows in particular that $\mfX$ is \'etale over $\mfX_i\times\G_m^k$ with $i$ being the largest index for which $\mfX\to\mfX_i$ is smooth. As $M(\mfX_{i\eta}\times\T^k)\cong M(\mfX_{i\eta})(k)[k]$ it suffices to show that $M(\mfX_{i\eta})$ lies in $\RigDA_{\gr}(L)$. In other words,  we may and do assume that $f_d$ is not smooth. By induction, we may then prove that if $(\Spf A)_{\eta}$ lies in $\RigDA_{\gr}(L)$ then also $(\Spf B)_{\eta}$ does if $\Spf B$ \'etale over $\Spf A\langle u,v\rangle/(uv-a)$. By excision and our induction hypothesis (see the last part of Step 3 in  \cite[Theorem 2.5.34]{ayoub-rig} where we remark that the ``induction on the dimension" is not used)  we may and do assume that $B=A\langle u,v\rangle/(uv-a)$. In this case, note that $M((\Spf A)_\eta)=M((\Spf A)_\eta\times\B^1)=M((\Spf A\langle u\rangle)_\eta)=M(\mfY_\eta)$ with $\mfY$ being the admissible formal blow up of $\Spf A\langle u\rangle$ on the ideal $(u,a)$. We note that $M(\mfY_\eta)\in\RigDA_{\gr}(L)$  sits inside a Mayer-Vietoris triangle with $M((\Spf B)_\eta)$, $M((\Spf A)_\eta\times\B^1)=M((\Spf A)_\eta)\in\RigDA_{\gr}(L)$ and $M((\Spf A)_\eta\times\T^1)=M((\Spf A)_\eta)(1)[1]\in\RigDA_{\gr}(L)$ proving the claim.
		\end{proof}

		\begin{prop}%
					Let $L$ be a finite unramified extension of $K_0$ with residue $k_L$.   
					There is a commutative diagram of motivic realization functors, compatible with field extensions:
					\be\begin{tikzcd}
						\logFDA^{\sst}(\mcO^\times_{L})\arrow[r, "\sim"] \arrow[d, "\xi"]& 	\logDA^{\sst}(k_L^{0})  \arrow[d, "R\Gamma_{HK}"] \\
						\RigDA_{\gr}(L)  \arrow[r, "R\Gamma_{\HK}^{}" ] & \mcD_{\varphi,N}(L_0)^{\op}
					\end{tikzcd}
					\ee	
					where the functor $R\Gamma_{\HK}$ on the right is induced by \Cref{dfn:HK} and the one below computes the arithmetic  overconvergent Hyodo--Kato cohomology of \cite[\S5.2.2]{CN19}. %
			\end{prop}

			\begin{proof}%
						Note that $\RigDA(C)=\varinjlim\RigDA_{\gr}(W(k')[1/p])$ by \Cref{prop:RigDAgr1}. The diagram then follows from \Cref{CN0} by unramified Galois descent and \cite[Proposition 5.13]{CN19}. 
					\end{proof}
					One may generalize the construction above to the case of  local fields of equi-characteristic as follows.
					\begin{dfn}\label{def:HKequi}
						We let $C$ be the completion of an algebraic closure of $k(\!(T)\!)$ and we define:
						\be
						R\Gamma_{\HK}^{{T}}\colon\logFDA^{\mathrm{ss}}(\mcO^\times_{C})\cong \logDA^{\mathrm{ss}}(\bar{k}^0)\stackbin{R\Gamma_{\HK}}{\longrightarrow}  \mcD_{\varphi,N}(K_0^{\unr})^{\op}
						\ee
						where semistable varieties are considered with respect to $k(\!(T)\!)$ [resp. $|T|$].
					\end{dfn}

					By definition, if $\mfX$ is a strictly semistable formal scheme over a finite unramified extension $L$ of $\Spf(k[\![T]\!])$, its Hyodo--Kato cohomology is defined as $R\Gamma_{\HK}(\mfX) := R\Gamma(\mfX_{\sigma}/W(k_L)^0)$ as $(W(k_L)[\frac{1}{p}], \varphi,N)$-module. In particular, it depends only on the log special fiber $\mfX_{\sigma}$ as log scheme over $k_L^{\log}$. Note that this is exactly the same formula used in the mixed characteristic case: the precise relationship is summarized in the following corollary.

					\begin{cor}\label{cor:commC}\begin{enumerate}
							\item
							The realization $R\Gamma_{\HK}^{p^\flat}\colon \log\FDA^{\sst}(\mcO_{C^\flat}^\times)\to  \mcD_{\varphi,N}(K_0^{\unr})^{\op}$ has rig-\'etale descent, and hence factors over a realization
							\be
							R\Gamma_{\HK}^{p^\flat}\colon \RigDA(C^\flat)\to  \mcD_{\varphi,N}(K_0^{\unr})^{\op}.
							\ee
							Moreover, the following diagram commutes:
							\be
							\xymatrix{
								&\logFDA^{\sst}(\mcO^\times_{C})\ar[r]&\RigDA(C)\ar[dl]^{R\Gamma_{\HK}}\\
								\logDA^{\sst}(\bar{k}^{0})\ar@{-}[ur]^{\sim}\ar@{-}[dr]_{\sim}\ar[r]^-{R\Gamma_{\HK}}&  \mcD_{\varphi,N}(K_0^{\unr})^{\op}\\
								&\logFDA^{\sst}(\mcO_{C^\flat}^\times)\ar[r]&\RigDA(C^\flat)\ar@{-}[uu]_{\sim}\ar[lu]_{R\Gamma^{p^\flat}_{\HK}}
							}
							\ee
							\item 
							Let $K$ be $k(\!(p^\flat)\!)$.    
							There is a commutative diagram of motivic realization functors, compatible with field extensions:
							\be
							\xymatrix{
								\logFDA^{\sst}(\mcO_{K}^\times)^{}\ar@{-}[r]^-{\sim}\ar[d]&
								\logDA^{\sst}(k^{0})^{}\ar@{-}[d]^{R\Gamma_{\HK}}\\
								\RigDA_{\gr}(K)^{}\ar[r]^-{R\Gamma_{\HK}^{p^{\flat}}} & \mcD_{\varphi,N}(W(k)[1/p])^{\op}
							}
							\ee
							where the functor $R\Gamma_{\HK}$ is induced by the  Hyodo--Kato cohomology of \Cref{dfn:HK}.%
						\end{enumerate}
					\end{cor}
					\begin{proof}
						For the first part, use \Cref{CN0} and \Cref{cor:locloc}. For the second, use once again that $\RigDA(C^\flat)\cong\RigDA_{\gr}(\bar{k}(\!(p^\flat)\!))$ and unramified Galois descent. The last claim is immediate from the definitions.
					\end{proof}
					We can extend the previous definitions to an arbitrary $K$, in accordance with \cite[Remark 4.13]{CN19}.
					\begin{dfn}\label{def:HKL}
						Let $L$ be a finite extension of $K$ [resp. $k(\!(T)\!)$]  with residue $k_L$, and let $L_0$ be the maximal  subfield of $L$ which is unramified over $K_0$ [resp. $k(\!(T)\!)$].  We denote by $R\Gamma_{\HK}$ the following   contravariant realization functor \be\RigDA_{\gr}(L)^{}\cong\RigDA_{\gr}(L_0)^{}\to \mcD_{\varphi,N}(K_0^{\unr})^{\op}\ee where the first equivalence is the one of \Cref{prop:RigDAgr1}. %
						 It is compatible with base change over $L$ by definition. 
					\end{dfn}
					The compatibility with tilting easily descends to the case of finite extensions. 
					\begin{cor}\label{cor:logtilt}\label{cor:commK}  
						Let $L$ be a finite extension of $K$ with residue $k_L$. Let ${L_\infty}$ be the completion of $L(p^{1/p^\infty})$.  The following diagram commutes
						\be
						\xymatrix{	&\logFDA^{\sst}(\mcO^\times_{L_0})^{}
							\ar[r]&\RigDA_{\gr}(L_0)^{}\ar[dl]^{R\Gamma_{\HK}} 
							\ar@{-}[r]^{\sim}&
							\RigDA_{\gr}(L_{\infty})^{}
							\\
							\logDA^{\sst}(k_{L_0}^{0})^{}\ar@{-}[ur]^{\sim}\ar@{-}[dr]_{\sim}\ar[r]^-{R\Gamma_{\HK}}&\mcD_{\varphi,N}(W(k_L)[1/p])^{\op}\\
							&\logFDA^{\sst}(\mcO_{k_L(\!(p^{\flat})\!)}^\times)^{}\ar[r]&\RigDA_{\gr}(k_L(\!(p^{\flat})\!))^{}\ar@{-}[uu]_{\sim}\ar[lu]_{R\Gamma^{p^\flat}_{\HK}}
							\ar@{-}[r]^{\sim}&
							\RigDA_{\gr}(L_{\infty}^\flat)^{}\ar@{-}[uu]_{\sim}
						}
						\ee\qed
					\end{cor}
					\begin{proof}
						In light of  \Cref{prop:comm}, \Cref{CN0}(2) and  \Cref{cor:commC}(2), only the triangle in the middle must be shown to be commutative. This follows from the Galois-invariance of the tilting equivalence $\RigDA_{\gr}(K_{\infty})\cong \RigDA_{\gr}(K^\flat_{\infty})$ (see e.g. \cite[Theorem 5.13]{LBV}), and   \Cref{cor:commC}(1).
					\end{proof}
					We also note the realizations above can be enriched with a $G_L$-structure, and extented to the whole of $\RigDA(L)$, as in the classical Hyodo--Kato setting.
					\begin{cor}
						The functor $R\Gamma_{\HK}$ can be extended to a functor
						\be
						\RigDA(L)^{\ct}\to \mcD^\flat_{\varphi,N,G_L}(K_0^{\unr})^{\op}
						\ee
						where  the category on the right is the derived DG-category of $(\varphi,N,G_L)$-modules \cite[Section 2.6]{DN}. 
					\end{cor}
					\begin{proof}As there is a natural transformation $\mcD^b_{\varphi,N}((-)_0)\to  \mcD^b_{\varphi,N,G_{(-)}}(K_0^{\unr})$ which takes the same values on $C$, and the functor on the right  has Galois descent, it suffices to apply the \'etale sheafification on the functor $R\Gamma_{\HK}$  of \Cref{def:HKL} (see \cite[Theorem 3.3.3(2)]{AGV}). 
					\end{proof}
                {
                \begin{rmk}
                    We can describe the extension above more explicitly, as follows. Every element in $\RigDA(L)^{\ct}$ lies in $\RigDA(L')^{\ct}_{\gr}$ after a sufficiently large finite base change $L'/L$, by means of \Cref{prop:RigDAgr1}. In particular, by Galois descent (see e.g. \cite[Theorem 2.3.4]{AGV}), we may identify the category $\RigDA(L)^{\ct}$ with $\varinjlim (\RigDA(L')^{\ct}_{\gr})^{\Gal(L'/L)}$. Using \Cref{def:HKL}, it then gains a canonical (contravariant) realization functor to $\varinjlim \mcD^b_{\varphi,N}(L'_0)^{\Gal(L'/L)}\simeq \mcD^\flat_{\varphi,N,G_L}(K_0^{\unr})$ as claimed.
                \end{rmk}}
					\begin{rmk}
The restrictions to $\RigDA(-)^{\ct}$ of the realization functors above			are monoidal. This is a formal consequence of \Cref{kunneth}.
					\end{rmk}
					\begin{rmk}%
						It is  clear from our definition that the Hyodo--Kato realization $R\Gamma^T_{\HK}$ induced on $\DA(k(\!(T)\!))$, with its structure as $(\varphi, N)$-module, coincides with  the  one considered by Lazda and Pal \cite[Chapter 5]{LP16} as they both have \'etale descent, excision and compare to Hyodo--Kato \cite[Theorem 4.27, Lemma 4.36, Theorem 5.46]{LP16}.  %
					\end{rmk}
					The following corollary is an ``arithmetic" version of \Cref{cor:HKtilt}. %
					\begin{cor}\label{cor:commK2}
						Let $K$ be a finite extension of $\Q_p$. Let $K_{\infty}$ be a totally ramified perfectoid extension of $K$, whose tilt $K_{\infty}^\flat$ is the completed perfection of a finite extension $K^\flat$ of $\F_p(\!(T)\!)$. If $Z$ is a smooth proper variety over $K^\flat$ with semi-stable reduction, then the   image of $\Q_{K^\flat}(Z^{\an}) $ under the composition below
						\be
						\RigDA_{\gr}(K^\flat)\cong\RigDA_{\gr}(K_{\infty}^\flat)\cong\RigDA_{\gr}(K_{\infty})\cong\RigDA_{\gr}(K)\to\mcD_{\varphi,N}(K_0)^{\op}
						\ee
						coincides with the  Hyodo--Kato cohomology of its special fiber.\qed
					\end{cor}
					\begin{proof}
						Use the diagrams of  \Cref{prop:comm}, \Cref{CN0}(2) and  \Cref{cor:commC}(2).
					\end{proof}

\section{Motivic existence of  tubular neighborhoods}%
\label{sec:tube}
We now prove a motivic analogue of \cite[Theorem 3.6]{huber-f2} (see also \cite{huber-f1} and \cite{huber-f3}). 
We start by proving that the motive of a smooth family of rigid analytic varieties is locally constant around each rational point of the base (\Cref{cor:tn}.\eqref{relative}). This follows from a ``spreading-out" (or ``continuity") property of rigid motives, proved in \cite{AGV}. From this, the existence of a tubular neighborhood that does not change the motive immediately follows. 
A similar trick is used in the proof of \cite[Theorem 4.46]{LBV}. 

We remark that {if} $S$ is smooth over $K$, around any $K$-rational point there is an open neighborhood which  is isomorphic to $\B^N_K$ (\cite[Theorem 2.1.5]{Ber93}).

\begin{prop}\label{cor:tn}%
	Let $X\to S$ be a qcqs smooth morphism  of rigid analytic varieties over a complete non-archimedean field $K$. Let $s$ be a $K$-rational point of $S$ and let $X_s$ denote the fiber of $X$ over it. 
	\begin{enumerate}
		\item\label{relative} For any sufficiently small open neighborhood $U$ of $s$, we have a canonical equivalence $\Q_S(X_s\times_K U)\cong\Q_S(X\times_SU)$ in $\RigDA(S)$.%
		\item\label{absolute} Assume that $S$ is smooth over $\Spa K$. Then, for any sufficiently small open neighborhood $U$ of $s$ that is isomorphic to $\B^N_K$, the natural morphism $\Q_K(X_s)\to\Q_K(X\times_SU)$ in $\RigDA(K)$ is invertible.
	\end{enumerate}
\end{prop}

\begin{proof}Since the statement of \eqref{relative} is local on $S$, we may assume that $X$ is qcqs. By \cite[Corollary 2.4.13 and Theorem 2.8.15]{AGV} the natural specialization maps $s^*$ induce an equivalence $\RigDA(s)^{\ct}=\RigDA(K)^{\ct}\cong\varinjlim\RigDA(V)^{\ct}$ in $\Prlo$ as $V$ runs among affinoid neighborhoods of $s$.  
	
	We notice that the pullback $\Pi^*$ along the structural morphism $\Pi\colon S\to\Spa K$ defines an explicit quasi-inverse of the functor above. This implies that for any compact motive $M$ in $\RigDA(S)$, its restriction to $U$ is canonically isomorphic to $\Pi^*s^*M$ in $\RigDA(U)$ for $U$ sufficiently small. If we apply this to $M=\Q_{S}(X)$ (which is compact by \cite[Corollary 2.4.13 and Proposition 2.4.22]{AGV}) we deduce $\Q_{U}(X_s\times_K U) \cong \Q_{U}(X\times_SU)$ or, equivalently $\Q_{S}(X_s\times_K U) \cong \Q_{S}(X\times_SU)$.

	To obtain \eqref{absolute}, we note that, by %
	 applying the functor $\Pi_\sharp $ (see \cite[Proposition 2.2.1(1)]{AGV} for the definition) to \eqref{relative}, we obtain a canonical equivalence $\Q_K(X_s\times U)\cong\Q_K(X|_U)$, whose composition with the equivalence $\Q_K(X_s)\cong\Q_K(X_s\times_K U)$ is the morphism $\Q_K(X_s)\to \Q_K(X\times_SU)$ defined by the canonical inclusion $X_s\to X\times_SU$. Thus, the assertion follows. 
\end{proof}

\begin{rmk}
Note that in \Cref{cor:tn}\eqref{absolute}	we could choose $U$ to be any  rigid analytic variety which becomes contractible in $\RigDA(K)$, e.g. a poly-disc with with radius in $\sqrt{|K|}$ \cite[Proposition 1.3.4]{ayoub-rig}.
\end{rmk}
\begin{rmk}\label{smoothness assumption}\label{rmk:hope}
	Note that we need the smoothness assumption contrary to Huber's result \cite[Theorem 3.6]{huber-f2}. %
	In the above proof, we use it to deduce that the motive $f_!f^!1$ is compact in $\RigDA(S)$ for the structural morphism $f$; this would not be true if $f$ had singular fibers (cf.\ \cite[Theorem 6.2.2]{huber}, \cite[(0.2)]{huber-f1}, and \cite[(0.1)]{huber-f2}). %
	However, we expect that the conclusion of \Cref{cor:tn}\eqref{absolute}  holds in more general situations. 
\end{rmk}

\begin{rmk}\label{ell independent tn}
	A version of \Cref{cor:tn} with torsion coefficients 
	recovers, under the smoothness assumption, a result of K.\ Ito: In \cite[Theorem 1.2]{Ito20}, it is proved that, in an algebraizable situation, we can $\ell$-independently take a tubular neighborhood that does not change the $\F_\ell$-coefficient \'etale cohomology, where $\ell$ is prime to the residual characteristic of $K$. 
	There, the proof uses the theory of nearby cycles over general bases, and is completely different from ours. 
\end{rmk}

{
\begin{rmk}
A result on the local existence of tubular neighborhoods  is already due to Kiehl \cite[Theorem 1.18]{kiehl}. Also, it has been observed and communicated to us by Gabber that one can  prove the following form of the statement (which is sufficient for our specific application): let $f\colon X\to \B^n$ be a smooth  morphism of affinoid adic spaces over $K$, %
then there exist an element $\e\in |K^\times|$ and a retraction $X|_{\B^n(\e)}\to X_0$ of $X_0\subset X|_{\B^n(\e)}$ such that the induced morphism $X|_{\B^n(\e)}\to X_0\times \B^n(\e)$ is an isomorphism. To prove this, we may and do pick a closed immersion $X\to \B^N$ over $\B^n$, and fix a section $s$ of the surjection $X_0\times_{\B^N}T\B^N\to T_{X_0}\B^N$ on the vector bundles. 
We consider the map $T_{X_0}\B^N\to \A^N$ defined by $(x,v)\mapsto x+s(v)$, which is \'etale on a neighborhood of $X_0$, and hence restricts to an isomorphism on an open neighborhood of the zero section (see e.g. \cite[Corollaire 1.1.56]{ayoub-rig}). 
Thus, by replacing $X$ with a neighborhood of $X_0$ we may assume that there exists a retraction $r\colon X\to X_0$. 
Then the map $(r,f)\colon X\to X_0\times\B^n$ is \'etale on a neighborhood of $X_0$, and hence induces an isomorphism on a neighborhood of $X_0$.  
\end{rmk}

}

We now focus on the main geometric object to which we will apply the above result. We refer to \cite[Section 8]{scholze} for details on toric varieties.

We fix a complete non-archimedean valued field $K$, a projective smooth toric variety $X_\Sigma$ over $K$ of dimension $n$,  {where $\Sigma$ is a fan in a fixed Euclidean space. Let $T$ be the torus acting on $X_\Sigma$. Recall from \cite[Section 4.1]{cox_little_schenck} (or \cite[Proposition 8.4]{scholze}) that for each 1-dimensional cone (i.e., a ray) $\tau_i$ of $\Sigma$ there is an open subset $U_{\tau_i} \simeq \A^1 \times \G_m^{n-1}$ of $X_\Sigma$, containing the closed $T$-invariant subset $\{0\}\times \G_m^{n-1}$. Its closure in $X_\Sigma$ is a $T$-invariant divisor on $X_\Sigma$, denoted $D(\tau_i)$. A \emph{$T$-Weil divisor} on $X_\Sigma$ is a Weil divisor of the form $D=\sum a_i D(\tau_i)$ with $a_i\in \Z$. It is well-known (see e.g. \cite[Theorem 4.1.3]{cox_little_schenck}) that every Weil divisor on $X_\Sigma$ is linearly equivalent to a $T$-Weil divisor. In particular, if $Y\subset X_\Sigma$ is a hypersurface, there is a $T$-Weil divisor $D$ equivalent to $Y$ and we can choose a section $f\in H^0(X_\Sigma, \mathcal{O}(D))$ whose zero locus is $Y$.

For $\varepsilon\in|K^*|$, let $Y(\e)$ denote the analytic open neighborhood of $Y^{\an}$ inside $X_\Sigma^{\an}$ introduced in the proof of \cite[Proposition 8.7]{scholze}, which is described (locally) by the inequalities $|f(x)|\leq\varepsilon$. More precisely, it is constructed as follows.
Note that since both $X_\Sigma$ and $D$ are defined combinatorially, they both admit natural models $\mathfrak{X}_{\Sigma, \mcO_K}$ and $\mathfrak{D}$ over {$\Spf(\mathcal{O}_K)$} (defined by the same toric data).  By choosing a cover by open formal subschemes $\mathfrak{U}$ of $\mathfrak{X}_{\Sigma, \mathcal{O}_K}$ {trivializing $\mc O(\mathfrak{D})$ and by fixing a trivialization $\mathcal{O}(\mathfrak{D}_{|\mathfrak{U}})\simeq \mc O_\mathfrak U$, we obtain functions $f_\mathfrak U \in H^0(\mathfrak{U}, \mathcal{O}_{\mathfrak{U}})$, } and thus we can consider the rational subsets of $\mathfrak{U}_K$ defined {by the inequalities $|f_\mathfrak{U}(x)|\leq\e$}
(which are independent of {the choice of trivializations}).  
As in \cite{scholze}, we will write $Y(\e)$ for their union.
}

{Suppose now that $Y\subset X_\Sigma$ is    a smooth subvariety} of codimension $c$ which is the scheme-theoretic intersection $Y_1\times_{X_\Sigma}\cdots\times_{X_\Sigma}Y_c$ of $c$ hypersurfaces $Y_i=V(f_i)$ with $f_i\in H^0(X_{\Sigma}, \mcO(D_i))$ a global section of some line bundle $\mcO(D_i)$. 
For $\varepsilon\in|K^*|$, {we can consider the analytic open neighborhood $Y(\e)$ of $Y^{\an}$ inside $X_\Sigma^{\an}$} defined (locally) by the inequalities $|f_i(x)|\leq\varepsilon$. 
\begin{cor}\label{hubermot2}
	Let $K$ be a complete non-archimedean field, $X_\Sigma$ be a projective smooth toric variety over $K$, and $Y$ be a smooth %
	scheme-theoretic intersection $Y_1\times_{X_\Sigma}\cdots\times_{X_\Sigma}Y_c$ of $c$ hypersurfaces $Y_i=V(f_i)$. 
	Then, there exists an element $\e_0\in |K^*|$ such that for any $\varepsilon<\e_0$, the natural morphism $\Q_K(Y)\to\Q_K(Y(\varepsilon))$ in $\RigDA(K)$ is invertible.
\end{cor}
\begin{proof}%
	{As above, following  the proof of \cite[Proposition 8.7]{scholze}}, we consider the natural formal model $\mf X_{\Sigma,\mc O_K}$ of $X_\Sigma$, and an integral model of $\mc O(D_i)$ by regarding $D_i$ as $T$-Weil divisors. 
	Since $\mf X_{\Sigma,\mc O_K}$ is quasi-compact, we can prove the statement {locally on $\mf X_{\Sigma,\mc O_K}$}. 
	
	On each open formal subscheme $\mf U$ of $\mf X_{\Sigma,\mc O_K}$ trivializing all the line bundle $\mc O(D_i)$, we fix trivializations of $\mc O(D_i)$, so that $(f_1,\ldots,f_c)$ gives a morphism $f\colon \mf U_K\to \B^c_K$. 
	As $Y$ is smooth, the morphism $f$ is smooth on a neighborhood of the fiber $f^{-1}(0)$. 
	Here note that the neighborhoods $f^{-1}(\B^c_K(\e))$, $\e\in |K^*|$, form a fundamental system of neighborhoods of $f^{-1}(0)$ (see for example \cite[Lemma 4.3]{Ito20}). 
	Thus, there exists an element $\e\in |K^*|$ such that $f$ is smooth over $\B^c_K(\e)$, for which we can apply \Cref{cor:tn},\eqref{absolute}. 
\end{proof}

\section{The $p$-adic weight-monodromy for complete intersections}\label{sec:proof}
We can now prove our main theorem. We first make some recollections on the notation we will use. 

Let $\bar{\Q}_p$ denote an algebraic closure of $\Q_p$ and $\C_p$ its completion. 
Let $\Q_p^\unr$ denote the maximal unramified extension of $\Q_p$ inside $\bar{\Q}_p$. 
\begin{dfn}
Let $w$ be an integer. 
\begin{enumerate}
\item Let $K_0$ be a finite unramified extension of $\Q_p$. We let $p^a$ be the cardinality of the residue field of $K_0$. 
We say that a finitely dimensional $\varphi$-module $D$ is \emph{pure of weight $w$} if every eigenvalue $\alpha$ of the $K_0$-linear endomorphism $\varphi^a$ is pure of weight $w$ in the sense of \cite[Definition 1.2.1]{Del80}. 
\item Let $D$ be a finitely dimensional $\varphi$-module over $\Q_p^\unr$. We can take a finite unramified extension $K_0$ of $\Q_p$ inside $\Q_p^\unr$ and a $\varphi$-module $D_0$ over $K_0$ together with an isomorphism $D_0\otimes_{K_0}\Q_p^\unr\cong D$ of $\varphi$-modules over $K_0^\unr$. 
We say that $D$ is \emph{pure of weight $w$} if $D_0$ is pure of weight $w$. This definition does not depend on the choices. 
\item Let $D$ be a finitely dimensional $(\varphi,N)$-module over $\Q_p^\unr$. Then $N$ is a nilpotent endomorphism, and hence it has the associated monodromy filtration $M_\bullet$ defined in \cite[Proposition 1.6.1]{Del80}. 
We say that $D$ is \emph{quasi-pure of weight $w$} if, for every integer $j$, the $j$-th graded quotient $\gr^M_jD$ of the monodromy filtration, which is a $\varphi$-module over $\Q_p^\unr$, is pure of weight $w+j$. 
\end{enumerate}
\end{dfn}

As noted in the introduction, we may give an equivalent formulation of the weight-monodromy conjecture: 

\begin{conj}
Let $X$ be a proper smooth algebraic variety over $\C_p$ defined over $\bar{\Q}_p$. Then the $(\varphi,N)$-module $H^i_{\HK}(X^{\an})$ is quasi-pure of weight $i$. %
\end{conj}

Let $\C_p^\flat$ denote the tilt of $\C_p$, which is the completion of an algebraic closure of the Laurent series field $F=\F_p\lcc p^\flat\rcc$ (\cite[Theorem 3.7]{scholze}). 
Let $\bar F$ %
denote the algebraic closure of $F$ inside $\C_p^\flat$. 
Let $R\Gamma_\HK^\flat$ denote the tilted Hyodo--Kato cohomology defined in \Cref{subsec:HK}. %

\begin{prop}\label{wmc for tilted HK}
Let $Z$ be a proper smooth variety over $\C_p^\flat$ defined over $\bar F$. 
Then the $(\varphi,N)$-module $H^{\flat,i}_\HK(Z^\an)$ satisfies the weight monodromy conjecture. 
\end{prop}
\begin{proof}
 We may assume that $Z$ is the base change of a proper smooth variety $Z'$ over a finite extension $E$ of $F$ inside $\bar F$. 
 Let $k$ denote the residue field of $E$.  By \cite[Theorem 6.5]{deJ96} and Poincar\'e duality (\Cref{poincare} and \Cref{dimax})  we may and do assume that $Z'$ has semistable reduction over $\mcO_E$ {(cf.\ \cite[Lemma 7.1.2]{ked-thesis}, see also \cite[Lemma 5.36]{LP16})}. 
 
Let $Z'_0$ denote the special fiber of a semistable model of $Z'$ over $\mc O_E$ equipped with the natural log structure.  
Then, by \Cref{cor:HKtilt}, the $(\varphi,N)$-module $H^{\flat,i}_\HK(Z^\an)$ is isomorphic to the $(\varphi,N)$-module $H^i_{\HK}(Z'_0/W(k)^0)\otimes_{W(k)[1/p]} \Q_p^\unr$ (with twisted monodromy operator). The latter is quasi-pure of weight $i$ by  \cite[Theorem 5.33, Theorem 5.46]{LP16}. 
\end{proof}

As in \cite[Section 8]{scholze}, we use the notation $X_{\Sigma,\kappa}$ for a toric variety $X_\Sigma$ over a field $\kappa$ to clarify the ground field  when necessary.  %
For a perfectoid field $L$, we use the notations $\mc X^\an_{\Sigma,L}$, $\mc X^\perf_{\Sigma,L}$ introduced in \cite[Paragraph after 8.4]{scholze} (there the notation $(-)^\ad$ is used instead of $(-)^\an$). 
Then as noted there, if $X_{\Sigma,L}$ is proper, then $\mc X^\an_{\Sigma,L}$ is identified with the analytification $X^\an_{\Sigma,L}$ of $X_{\Sigma,L}$. 
Recall from \cite[Theorem 8.5]{scholze} that the tilt of $\mc X^\perf_{\Sigma,L}$ is canonically isomorphic to the perfectoid space $\mc X^\perf_{\Sigma,L^\flat}$ for %
the same fan $\Sigma$ defining $X_{\Sigma,L}$. 

The following is a motivic analogue of \cite[Proposition 8.6]{scholze}. 
\begin{lemma}
Let $L$ be a perfectoid field and $X_{\Sigma,L}$ be a smooth toric variety. 
Then we have $\Q_{L^\flat}(\mc X^\an_{\Sigma,L^\flat})^\sharp\cong \Q_{L}(\mc X^\an_{\Sigma,L})$, where $(-)^\sharp$ denotes the equivalence  $\RigDA(L^\flat)\stackrel{\sim}{\to}\RigDA(L)$ from \cite[Theorem 7.26]{vezz-fw}.
\end{lemma}
\begin{proof}
	Via the equivalence $\RigDA(L^\flat)\cong\PerfDA(L^\flat)=\PerfDA(L)$ of \cite[Theorem 6.9]{vezz-fw} (see the proof of \cite[Theorem 3.8]{BV} on how to omit the $\Frob$-localization) the motive of $\mcX^{\an}_{\Sigma,L^\flat}$ is sent to that of $\mcX^{\perf\sharp}_{\Sigma,L^\flat}$. Note that $\mc X^\an_{\Sigma,L}$ has a  smooth formal model $\mf X_{\Sigma,\mc O_L}$ over $\mcO_{L}$ and that $\varprojlim_\varphi \mc X^\an_{\Sigma,L}$ is (locally) a presentation of good reduction in the sense of \cite[Definition 2.4]{vezz-fw}. %
	 Then, by \cite[Proposition 5.4]{vezz-fw}, we deduce that the projection gives an equivalence  $\Q_{L}(\mc X_{\Sigma,L^\flat}^{\perf\sharp})=\Q_{L}(\varprojlim_\varphi \mc X^\an_{\Sigma,L})\cong \Q_{L}(\mc X^\an_{\Sigma,L})$.
\end{proof}

\begin{cons}\label{map of motives}
Let $L$ be a perfectoid field and $X_{\Sigma,L}$ be a smooth toric variety. 
Recall from \cite[Theorem 8.5.(iii)]{scholze} that we have a natural continuous map $\pi\colon|\mc X^\an_{\Sigma,L^\flat}|\cong|\mc X^{\perf}_{\Sigma,L}|\to|\mc X^\an_{\Sigma,L}|$. 
Let $V\subset \mc X_{\Sigma,L}$ be an open adic subspace. 
We consider the inverse images $V^\perf\subset \mc X^\perf_{\Sigma,L}$ and $\pi^{-1}(V)\subset \mc X^\an_{\Sigma,L^\flat}$ viewed as open adic subspaces.  
	Then the morphism %
	$V^{\perf}\to V$ of adic spaces induces a natural morphism %
	$j^*\Q(V^{\perf})\to \iota^*\Q(V)\to \iota^*\Q(\mc X^\an_{\Sigma,L})$ in $\wRigDA(L)$ (with the notation in \cite[Page 40]{vezz-fw}), and hence the following commutative diagram in $\RigDA(L)$
	\be
	\xymatrix{
		\Q(\mc X_{\Sigma,L^\flat}^\an)^\sharp\ar@{=}[r] & \Q(\mc X^\an_{\Sigma,L})\\
		\Q(\pi^{-1}(V))^\sharp\ar[u]\ar[r] &  \Q(V).\ar[u]\\
	}
	\ee
\end{cons}

We have now all the ingredients to adapt %
Scholze's %
proof to the $p$-adic setting. %

\begin{thm}\label{thm:main}
Let $\C_p$ be the completion of an algebraic closure of $\Q_p$ 
	and let $Y$ be a   smooth variety over $\C_p$ which is a scheme-theoretic complete intersection inside a projective smooth toric variety $X_{\Sigma,\C_p}$. 	Then the $p$-adic weight monodromy conjecture holds for $Y$, that is, the $(\varphi,N)$-module $H^i_\HK(Y^\an)$ is quasi-pure of weight $i$.
\end{thm}
\begin{proof}%
By \Cref{hubermot2}, there is an open neighborhood $\widetilde{Y}$ of $Y^\an$ inside $X^\an_{\Sigma,\C_p}$ whose motive coincides with that of  $Y^\an$. We let $\pi^{-1}(\widetilde{Y})\subset X^\an_{\Sigma,\C_p^\flat}$ be the inverse image of $\widetilde{Y}$ via the continuous map $\pi\colon|X^\an_{\Sigma,\C_p^\flat}|\to|X^\an_{\Sigma,\C_p}|$. %
By  \cite[Corollary 8.8]{scholze}, we can find an irreducible closed algebraic subvariety $Z$ of $X_{\Sigma,\C_p^\flat}$ defined over the algebraic closure of $\F_p\lcc p^\flat\rcc$ in $\C_p^\flat$ such that $\dim Z=\dim Y$ and its analytification lies in $\pi^{-1}(\widetilde{Y})$. 
We use \cite[Theorem 4.1]{deJ96} to take a smooth alteration $Z'$ of $Z$. %

	Then \Cref{map of motives} gives us the following commutative square in $%
	\RigDA(\C_p)$
\be
	\xymatrix{
		\Q(X_{\Sigma,\C_p^\flat}^{{\an}})^\sharp\ar@{=}[r] & \Q(X_{\Sigma,\C_p}^{{\an}})\\
		\Q({Z'}^\an)^\sharp\ar[u]\ar[r]& \Q(Y^\an).\ar[u]
	}
	\ee
	Applying the realization $R\Gamma_{\HK}$, we obtain a natural morphism
	\be
	\alpha\colon R\Gamma_\HK(Y^\an)\to R\Gamma_\HK^\flat({Z'}^\an)
	\ee
	in $\mcD_{\varphi,N}({\Q_p^\unr})$. %
Let $d$ denote $\dim Y=\dim Z'$ and for an integer $i$, let $\alpha^i$ denote the map on cohomology $H^i_{\HK}(Y^{\an})\to H^{\flat, i}_{\HK}(Z'^{\an})$ induced by $\alpha$. By \Cref{nonvanishing} and \Cref{nonvanishing2}, the canonical map $H^{2d}_\HK(X^\an_{\Sigma,\C_p})\cong H^{\flat,2d}_{\HK}(X^\an_{\Sigma,\C_p^\flat})\to H^{\flat,2d}_{\HK}({Z'}^\an)$ can't be zero. In particular, the map $\alpha^{2d}\colon H^{2 d}_{\HK}(Y^\an)\to H^{\flat,2 d}_{\HK}({Z'}^\an)$ is not zero. Since   both sides of $\alpha^{2d}$ are one-dimensional (\Cref{dimax}) we deduce that $\alpha^{2d}$ is an isomorphism.

By Poincar\'e duality (\Cref{poincare}) we deduce that the dual of $\alpha^{2d-i}$ with respect to $H^{2d}_\HK(Y^\an)\cong H^{\flat,2d}_\HK({Z'}^\an)$ gives a $(\varphi,N)$-equivariant splitting of $\alpha^i$, and hence $H^i_{\HK}(Y^\an)$ is a direct summand of $H^{\flat,i}_{\HK}({Z'}^\an)$. Thus, the assertion follows from  \Cref{wmc for tilted HK}. 
\end{proof}

\begin{rmk}
	By taking the motivic $\ell$-adic realization instead of the $p$-adic one, the proof above coincides with (a motivic version of) Scholze's proof of the $\ell$-adic weight monodromy conjecture for scheme-theoretical complete intersections in toric varieties. For a motivic rigid analytic $\ell$-adic realization functor, one can use \cite{BV}.
\end{rmk}

\begin{rmk}\label{geoconn}
	If $Y$ is a   complete intersection of positive dimension in a smooth and projective toric variety $X_\Sigma$, then $Y$ is geometrically connected. This can be deduced from the Grothendieck-Lefschetz theorem \cite[Corollaire XII.3.5]{SGA2} (we recall that a complete intersection is Cohen-Macaulay by e.g. \cite[Lemma 00SB]{stacks-project}). 
\end{rmk}

\appendix
\section{A log-free Hyodo--Kato cohomology}\label{sec:bond}
In this appendix, we discuss a motivic way to equip %
the {$K_0$-}overconvergent de Rham realization with a $(\varphi,N)$-structure %
which does not make any reference to the Hyodo--Kato isomorphism or log structures\footnote{{An expanded version of this section can be found in \cite{BGV}}}. 
This procedure can be done in the $\ell$-adic realization as well, see \cite[Section 11]{ayoub-etale}.%

The basic idea is to note that the graded pieces with respect to the weight filtration on the de Rham realization factors over the motivic nearby cycle functor $\Psi$, giving a formula $\gr_k^{\omega}H^*_{\dR,K_0}(X)\cong\gr_k^\omega H^*_{\rig}(\Psi(X))$. Since the functor $\Psi$ is equipped with a monodromy operator $N$, this is also reflected on the (weight filtration of the) cohomology groups.  This is in accordance with the formulas of \cite{mokrane}. %

Let the notations be as in the beginning of \Cref{subsec:HK}. 
In particular, $K$ is a complete discrete valuation field of mixed characteristic with perfect residue field $k$ and $K_0$ is the subfield $W(k)[1/p]$. 
As in \Cref{K_0 str}, we consider the $K_0$-overconvergent de Rham realization $R\Gamma_{\dR,K_0}\colon\RigDA_\gr(K)\cong \RigDA_\gr(K_0)\to \mc D(K_0)$. 
We propose to define the $(\varphi,N)$-structure on it purely in terms of the generic fiber. 
In what follows, we will construct a Frobenius structure on the complex and a monodromy operator on the level of graded pieces of the weight filtration on each cohomology group. 

\subsection{Motivic nearby cycles}
We first recall how the motivic nearby cycle functor $\Psi$ is defined on rigid analytic motives (of good reduction) following Ayoub. Recall that $\RigDA_{\gr}(K)$ is canonically equivalent to the category of $\chi1$-modules in $\DA(k)$ (\cite[Theorem 3.3.3(1)]{AGV}) and that it  contains the motives of varieties with pluri-nodal reduction see \Cref{ssisgr}.  
We choose $p$ as a uniformizer of $K_0$ to fix an identification $\chi1\cong1\oplus1(-1)[-1]$, i.e., via $\chi1\cong\chi_01$.  
We consider the augmentation morphism $\chi1\cong1\oplus1(-1)[-1]\to1$ in $\DA(k)$ corresponding to $(1\,\, 0)$, which is the only one algebra morphism. In light of \cite[Scholie 1.3.26]{ayoub-rig}, we give the following definition.
\begin{dfn}
The motivic nearby cycle is the functor \be\Psi\colon\RigDA_{\gr}(K)\cong\Mod_{\DA(k)}(\chi1)\to\Mod_{\DA(k)}(1)\cong\DA(k).\ee induced by the canonical augmentation $\chi1\cong1\oplus1(-1)[-1]\to1$. 
\end{dfn}

\begin{prop}\label{motivic nearby cycle}
If $M$ lies in $\RigDA_{\gr}(K)^{\ct}$, the object $\Psi M$ is canonically equipped with a morphism $N\colon \Psi M\to \Psi M(-1)$ in $\DA(k)$. {More precisely,} the functor $\Psi$ factors as
		\be\Psi\colon\RigDA_{\gr}(K)^{\ct}\cong\Mod_{\DA(k)}(\chi1)^{\ct}{\to}\DA(k)_N^{\ct}\to\DA(k)^{\ct},\ee
		where the category $\DA(k)_N$ is informally given by ind-nilpotent maps  $N\colon M\to M(-1)$ i.e.,  comodules over the coalgebra $\bigoplus\Q(-n)$\footnote{See \cite[\S 2]{BGV}}.
\end{prop}

\begin{proof}As in \Cref{rmk:UDA} we let $q\colon\G_{m,k}\to\Spec k$ be the structural morphism and  $\UDA(k)$ be the full subcategory of $\DA(\G_{m,k})$ generated under colimits by motives of the form $q^*M$ for $M\in\DA(k)$. It can be identified with the category $\Mod_{q_*1}(\DA(k))$ and hence with $\RigDA_{\gr}(K)$ (see \Cref{rmk:UDA}). The restriction of the (monoidal) base-change functor $1^*\colon\DA(\G_m)\to\DA(k)$ to $\UDA(k)^{\ct}$ is the functor induced by the algebra morphism $q_*1\to1$ in $\DA(k)$ so it coincides with the functor $\Psi$ above, modulo the above identifications.

By the last part of the proof of \cite[Scholie 1.3.26.2]{ayoub-rig} we may also identify it with the functor $\Upsilon_{\id}$ (see, for example, \cite[Section 10]{ayoub-etale} for the definition). 
The existence of the monodromy operator follows then from  %
 \cite[Th\'eor\`eme 11.16]{ayoub-etale}.%
\end{proof}

\begin{rmk}\label{eK2}
	Had we chosen to identify $\chi1$ with $1\oplus1(-1)[-1]$ via a uniformiser of $K$, then $N$ would have been multiplied by the constant $e_K$ (see \Cref{eK1}).
\end{rmk}

\begin{rmk}
		{In \cite[Theorem 1.7]{BGV},	it is shown that} that the $K_0$-overconvergent de Rham realization $R\Gamma_{K_0,\dR}$ on $\RigDA_{\gr}(K)$ factors as $R\Gamma_\rig(-/K_0)\circ\Psi$, and hence the monodromy operator on the motivic nearby cycle induces a canonical monodromy operator on $R\Gamma_{K_0,\dR}$  (see more precisely \Cref{rmk:future}). We now prove this fact on the graded pieces of the weight filtration. 
\end{rmk}

\subsection{Weight structures}
		We shall construct the monodromy operator on the graded pieces of the weight filtration of each cohomology group via a factorization over the nearby cycle functor. To make this precise, 
		we recall the following definition, initially due to Bondarko and Pauksztello, then rephrased in the infinity-language by Sosnilo.
		\begin{dfn}[{\cite[Definition 1.13, Remark 1.14.1]{Sosnilo}}]
			Let $\mcC$ be a [closed symmetric monoidal] stable infinity-category. 	A \emph{[closed symmetric monoidal] bounded weight structure of $\mcC$ }is a choice of a full subcategory   $\mcH$ (the \emph{heart}) of $\mcC$ such that
			\begin{enumerate}
				\item it is closed under direct summands [and tensor products] in $\mcC$;
				\item it generates $\mcC$ under finite limits (or colimits);
				\item  $\map(X,Y)$ is {connective} for all $X,Y\in\mcH$ i.e. $\pi_i\map(X,Y)=0$ if ${i<0}$.
			\end{enumerate} 
			A functor between categories $\mcC\to\mcC'$ equipped with a {[closed symmetric monoidal] bounded weight structure} is called \emph{weight-exact }if the image of $\mcH$ lies in $\mcH'$.
		\end{dfn}

\begin{exm}\label{eg:ws}
		
			The following are examples of closed monoidal bounded weight structures.
			\begin{enumerate}
				\item %
				For an additive category $\mc A$, the infinity-category $\Ch^\flat(\mc A)$ of bounded chain complexes together with the full subcategory $\mc A$ embedded as complexes concentrated in degree zero is a  bounded weight structure by \cite[Section 1.3]{Sosnilo}. 
				\item {Let $\kappa$ be a field.} The Karoubi-closure $\mcP$ of the category of Chow motives, {i.e., objects in $\DA(\kappa)$ }of the form ${\Q}_\kappa(X)(i)[2i]$ with $i\in\Z$ and $X/\Spec \kappa$ smooth and proper  is a symmetric closed monoidal bounded weight structure  in $\DA(\kappa)^{\ct}$ by \cite[Proposition 6.5.3]{bondarko}.
				\item\label{UDA} {Let $K$ be a non-archimedian field.} The Karoubi-closure $\mcP_{K}$ of the category of  motives of the form $\xi M$ with $M\in\mcP$  defines a symmetric closed monoidal bounded weight structure on $\RigDA_{\gr}(K)^{\ct}$. To see this, notice that if $A,B$ are in $\mcP$ then $\map(\xi A,\xi B)\cong\map(A,B)\oplus\map(A,(B(-1)[-2])[1])$ (this follows from the identification of $\chi1$ given in the proof of \Cref{motivic nearby cycle}). The object $B(-1)[-2]$ lies in $\mcP$ so that the complex on the right is {connective}. %
			\end{enumerate}
		\end{exm}
		Let $(\mc C,\mc H)$ be a bounded weight structure. By \cite[Proposition 3.3.2]{Sosnilo} and \cite[Proposition 3.2.3]{Sosnilo2} (see also \cite{aoki}), the functor $\mcH\to\Ho\mcH$ induces a  [monoidal] weight-exact functor $
		\omega\colon \mcC%
		\to%
		\Ch^\flat(\Ho\mcH)
		$, which is unique up to homotopy.
		\begin{dfn} The functor $\omega$ is called the weight-complex functor. For an object $X\in\mcC$, we write
			\be
			\omega(X)= \left(\ldots\to X^{(-1)}\to X^{(0)}\to X^{(1)}\to \ldots \right),
			\ee
			where each $X^{(i)}$ is in $\mcH$.  
		\end{dfn}
		
There are situations where the weight complex is enough to recover the cohomology theory, equipping it with a weight filtration, as the following proposition by Bondarko \cite{bondarko} shows (also relevant \cite[Section 3]{Vol13}).
				\begin{prop}[{\cite[Theorem 2.4.2 and Remark 2.4.3(1)]{bondarko}}]
			\label{bondw}
			Let $A$ be an abelian category, $(\mcC,\mcH)$ be a bounded weight structure, and $R\Gamma\colon \mcC\to\mcD(A)^{\op}$ be a functor. 
			\begin{enumerate}
			\item There is a  spectral sequence, {functorial from page $E_2$}
			\be
			E_1^{p,q}=H^q(R\Gamma(X^{(-p)}))\Rightarrow H^{p+q}R\Gamma(X)
			\ee
			whose differentials at page $E_1$ are defined by the differentials of the weight complex. 
			\item We assume that $A$ is equipped with abelian subcategories $\{A_{\omega=i}\subset A\}_{i\in\Z}$ that are closed under subquotients and orthogonal and that $H^i(R\Gamma(M))\in A_{\omega=i}$ for each $M\in\mcH$. 			Then the above spectral sequence degenerates at $E_2$, and its induced graded pieces $\gr^\omega_iH^n$ lie in $A_{\omega=i}$. 

			\end{enumerate}
		\end{prop}
			
			In particular, the bi-graded object $\bigoplus\gr_iH^n(R\Gamma(X))$ only depends on $\omega(X)$. 
			More precisely, the functor $\gr^\omega_*H^*R\Gamma\colon \mcC\to %
			\gr_{\Z\times\Z}(A)$ factors as 
			\be
			\mcC\stackrel{\omega}{\to}\Ch(\Ho\mcH)\to  \gr_{\Z\times\Z}(A),
			\ee
			where the second functor is given by 
			\be
			M^\bullet\mapsto \bigoplus_{n,i}H^{n-i}(\cdots\to H^iR\Gamma(M^{-p})\to H^iR\Gamma(M^{-p-1})\to \cdots).\ee 

\begin{exm}\label{Frobenius on dR}\label{mwss for dR}
Let $K$ be a complete discrete valuation field with perfect residue field. 	The Frobenius-enrichment $F\colon \DA(k)\to\DA(k)^{h\varphi}_\omega $ on the category $\DA(k)$ (see \Cref{rmk:laxfixedpoints}) induces
	\be
	\Mod_{\chi_01}(\DA(k)^{\ct})\to \Mod_{F(\chi_01)}(\DA(k)^{\ct,h\varphi})
    \cong
    \RigDA_{\gr}(K_0)^{\ct,h\varphi^\sharp},
	\ee
	where 
    {$\varphi^\sharp$ is the automorphism on $\RigDA_{\gr}(K_0)$ obtained via a tilting equvalence $\RigDA_{\gr}(K_0)\simeq \RigDA_{\gr}(k(\!(T)\!))$ and the Frobenius on $k(\!(T)\!)$ {(see \cite[Proposition 4.49]{BGV}}). }
	By applying  the {de Rham} realization  we then obtain {(see \cite[Proposition 4.52]{BGV})}
	\be
	R\Gamma_{\dR,K_0}\colon \RigDA_{\gr}(K)^{\ct}\to \RigDA_{\gr}(K_0)^{\ct,h\varphi^\sharp}
    \to \QCoh(\Spa K_0)^{\ct,h\varphi}\cong\mcD_\varphi(K_0)^{\ct}.
	\ee
		When the residue field $k$ of $K$ is finite, this satisfies the hypotheses of \Cref{bondw}.(2): 
if $Y$ is a smooth and proper variety over $k$, the $i$-th cohomology group $H^i_{\dR,K_0}(\xi \Q_{K_0}({Y}))=H^i_\rig(Y/K_0)$ is $\varphi$-pure of weight $i$. Thus, for any $X\in\RigDA_{\gr}(K)^{\ct}$, each $H^i_{\dR,K_0}(X)$ comes equipped with a functorial weight filtration. 		\end{exm}
\subsection{Motivic monodromy operator}
Let $K$ be a non-archimedean field.	The Monsky--Washnitzer functor $\xi\colon\DA(k)^{\ct}\to \RigDA(K)_{\gr}^{\ct}$ {is} weight-exact by definition. More precisely, we have:
		\begin{prop}\label{xi wex}
		The functors $\Ho \mc P\to \Ho \mc P_K$ and $\Ho \mc P_K\to \Ho\mc P$ induced by $\xi$ and $\Psi$ are quasi-inverse to each other. In particular, we have a commutative square of monoidal functors
	\be
\xymatrix{
	\RigDA_{\gr}(K)^{\ct} \ar[r]\ar[d]^{\Psi} & \Ch^\flat(\Ho\mcP_K)\ar[d]^{\Psi}_{\sim}\\
	\DA(k)^{\ct}\ar[r]& \Ch^\flat(\Ho\mcP).
}
\ee%
		\end{prop}
		\begin{proof}
			As already remarked in \Cref{eg:ws}\eqref{UDA}, for any $A,B$  in $\mcP$ there is an equivalence  $\Hom(\xi A,\xi B)\cong\Hom(A,B)$. This shows that $\xi\mcP$ is Karoubi-closed and that   $\Ho\mcP\cong\Ho\mcP_{K}$ via $\xi$. This proves the statement  since $\Psi\xi\cong\id$.%
		\end{proof}
				\begin{rmk}
		If $K$ is algebraically closed,	the analytification functor $\DA(K)^{\ct}\to\RigDA(K)^{\ct}$ is \emph{not} weight exact with respect to the structures described above. 
		\end{rmk}

	We now use Propositions \ref{motivic nearby cycle} and \ref{xi wex} to endow the weight complex of an object $X$ in $\RigDA_{\gr}(K)^{\ct}$ with a monodromy operator. 

	We define the Tate twist functor $C^\bullet\mapsto C^\bullet(-1)$ on $\Ch^b(\mc P)$ by $C^\bullet(-1)=(C(-1)[-2]^\bullet )[2]$, more precisely, the $p$-th component of $C^\bullet(-1)$ is given by $C^{p+2}(-1)[-2]$. We remark that if $C$ lies in $\mcP$, then so does $C(-1)[-2]$. Then, for any object $M$ in $\RigDA_\gr(K)$, we have a functorial isomorphism $\omega(M(-1))\cong \omega(M)(-1)$. 
	\begin{cor}\label{N on weight complex}
For any object $M$ in $\RigDA_\gr(K)^{\ct}$, its weight complex $\omega(M)$ is canonically equipped with a monodromy operator $\omega(M)\to\omega(M)(-1)$, %
{i.e.,} the weight complex functor factors as
\be
\RigDA_{\gr}(K)^{\ct}\to (\Ch^\flat(\Ho\mcP_K))_N\to \Ch^\flat(\Ho\mcP_K).
\ee
	\end{cor}
\begin{proof}
	This follows from the fact that the left arrow in the commutative square from Proposition \Cref{xi wex} factors over $\DA(k)_N$.
\end{proof}
\begin{rmk}
As $C^\bullet$ is bounded, any monodromy operator $C^\bullet\to (C(-1)[-2]^\bullet)[2]$ is nilpotent. 
\end{rmk}

Finally, we apply the above observations to $K_0$-overconvergent de Rham cohomology $R\Gamma_{\dR,K_0}\colon \RigDA_\gr(K)\to \mc D_{\varphi}(K_0)$. 
By \Cref{N on weight complex}, the realization 
\be
\RigDA_{\gr}(K)^{\ct}\to \mcD_\varphi(K_0)^{\ct}\to \gr_{\Z\times\Z}\Mod_\varphi(K_0)
\ee
from \Cref{mwss for dR} canonically factors through $(\Ch^\flat(\mcP))_N$. 
More concretely, for any object $X$ in $\RigDA_\gr(K)$, the monodromy operator $\omega(X)\to \omega(X)(-1)$ from \Cref{N on weight complex} induces a (functorial) monodromy operator on the bi-graded object $\bigoplus_{n,i}\gr_iH^n_{\dR,K_0}(X)$: we obtain, by taking $(p+2)$-th components, a morphism $X^{(-p-2)}\to X^{(-p)}(-1)[-2]$ in $\mc P$; by taking the $(q-2)$-th rigid cohomology groups 
		\be
	N\colon  H^q_{\rig}(X^{(-p)})(1)\to H^{q-2}_{\rig}(X^{(-p-2)}),
		\ee
inducing a (nilpotent) monodromy operator $N\colon \gr^\omega_i(H_{\dR,K_0}^n(X))\to \gr^\omega_{i-2}(H_{\dR,K_0}^n(X))(1)$. 
In the semistable reduction case, the spectral sequence above is expected to be compared, in an appropriate sense, with the classical weight spectral sequence considered by Mokrane \cite{mokrane} (see also Nakkajima \cite{Nakkajima}) and Gro\ss e-Kl\"onne \cite{GK05} { (compare to \cite[\S 4.9]{BGV})}.

	\begin{rmk}
The Frobenius structure on $\RigDA(K)$ can be encoded more geometrically via the equivalence $	\RigDA(K) \cong \RigDA(\mathrm{Spd} K)$ of \cite[Theorem 5.13]{LBV}.
	\end{rmk}
		
		\begin{rmk}
Note that the formalism of weight-structures gives a positive answer to the ``generalized'' weight-monodromy conjecture in equi-characteristic $p$	of Lazda--Pal \cite[Conjecture 5.55]{LP16} by taking as ``geometric'' weight filtration the filtration on cohomology induced by the weight-complex.
		\end{rmk}
		
		\begin{rmk}\label{motmon}\label{rmk:future}
The category $\RigDA_{\gr}(K)\cong\Mod_{\chi1}(\DA(k))$ is {actually} equivalent to the category ${ \DA(k)_N}$ (see \cite[Corollary 4.14]{BGV}) i.e.,   its compact objects are given by ``motivic monodromy maps'' $M\to M(-1)$ in $\DA(k)$ (which are necessarily nilpotent, using weight considerations). %
			This  {implies} that any monoidal realization $\RigDA_{\gr}(K)\to\mcC$ can be enriched with a monodromy operator $\RigDA_{\gr}(K)\to\mcC_N$, {which recovers the usual monodromy operator,} and that $\Psi$ corresponds to the functor ``forgetting'' monodromy. This is an infinity-version of the constructions above, which 
            {is considered in \cite{BGV}.}%
		\end{rmk}

\end{document}